\documentclass[11pt,a4paper,twoside]{article}
\usepackage[UKenglish]{babel}
\usepackage[T1]{fontenc}
\usepackage[utf8x]{inputenc}
\usepackage{latexsym,amsfonts,amsmath,amsthm,amssymb, mathrsfs}
\usepackage{fullpage}
\usepackage{xcolor,bm, bbm}
\usepackage{paralist}

\usepackage{fourier}

\newcommand{\OBbeta}{ \mathbb{B}_R^{n,V }}
\newcommand{\R}{\mathbb R}
\newcommand{\N}{\mathbb N}
\newcommand{\eps }{ \varepsilon }

\newcommand{\DistkKompOrl }{ \lambda_{R, V}^{( n , k )}}
\newcommand{ \CondDistVoneVtwo}{ \widehat{\lambda}^{( R , k)}_{n , V_1 | V_2}( \cdot ) }
\newcommand{ \CondDistVoneVtwoeps}{ \widehat{\lambda}^{( R+ \epsilon  , k)}_{n , V_1 | V_2 }}

\DeclareMathOperator{\E}{\mathbb{E}}
\DeclareMathOperator{\prb}{\mathbb{P}}
\DeclareMathOperator{\I}{ \mathbb{I} }
\DeclareMathOperator{\U}{ \text{Unif} }
\DeclareMathOperator{\vol}{vol}

\newtheorem{thm}{Theorem}[section]
\newtheorem{cor}[thm]{Corollary}
\newtheorem{df}[thm]{Definition}
\newtheorem{proposition}[thm]{Proposition}
\newtheorem{rem}[thm]{Remark}
\newtheorem{lem}[thm]{Lemma}

\newtheorem{thmalpha}{Theorem}

{
\theoremstyle{definition}

}

\allowdisplaybreaks
\begin{document}

\title{\bf Sanov-type large deviations and conditional limit theorems for high-dimensional Orlicz balls}
\medskip

\author{Lorenz Fr\"uhwirth and Joscha Prochno}



\date{}

\maketitle

\begin{abstract}
	\small
	In this paper, we prove a Sanov-type large deviation principle for the sequence of empirical measures of vectors chosen uniformly at random from an Orlicz ball. From this level-$2$ large deviation result, in a combination with Gibbs conditioning, entropy maximization and an Orlicz version of the Poincar\'e-Maxwell-Borel lemma, we deduce a conditional limit theorem for high-dimensional Orlicz balls. Roughly speaking, the latter shows that if $V_1$ and $V_2$ are Orlicz functions, then random points in the $V_1$-Orlicz ball, conditioned on having a small $V_2$-Orlicz radius, look like an appropriately scaled $V_2$-Orlicz ball. In fact, we show that the limiting distribution in our Poincar\'e-Maxwell-Borel lemma, and thus the geometric interpretation, undergoes a phase transition depending on the magnitude of the $V_2$-Orlicz radius.
	\medspace
\end{abstract}

\section{Introduction \& Main results}

It is a classical and famous result, independently observed by Borel \cite{borel1914} and Maxwell (e.g., \cite{FreedmanDiaconis2011}), and commonly referred to as the Poincar\'e--Maxwell--Borel lemma, that any fixed number of coordinates of random vectors uniformly distributed on a high-dimensional Euclidean sphere is approximately Gaussian distributed. More precisely, if $X^{(n)}=\big(X^{(n)}(1),\dots,X^{(n)}(n)\big)$ is uniformly distributed on $\sqrt{n} \mathbb S^{n-1}= \big\{(x_i)_{i=1}^n\in\R^n\,:\,\sum_{i=1}^n |x_i|^2 = n \big\}$, then, for any fixed $k\leq n$, the distribution of the first $k$ coordinates $\big(X^{(n)}(1),\dots,X^{(n)}(k)\big)$ converges weakly to the distribution of a $k$-dimensional standard Gaussian random vector. The normalization of the sphere is taken to ensure that a typical coordinate of an element of $\sqrt{n} \mathbb S^{n-1}$ has unit order when the dimension $n$ of the ambient space is large. 

In 1991, Mogul'ski\u{\i} \cite{MogFin1991} and Rachev and R\"uschendorf \cite{RSR1991} obtained an $\ell_p$ analogue of the Poincar\'e--Maxwell--Borel lemma for the surface and cone measure respectively on a properly scaled $\ell_p$-sphere (note that surface and cone measure coincide whenever $p\in\{1,2,\infty\}$). A simplification of the arguments can be found in the work of Naor and Romik \cite[Theorems 3 and 4]{NR2003} and a significant generalization (including the case of Orlicz balls) has recently been obtained by Johnston and Prochno in \cite{PJ2020OrliczLim}. In the $\ell_p$-versions of the Poincar\'e--Maxwell--Borel lemma, instead of the standard Gaussian distribution, the so-called $p$-generalized Gaussian distribution appears in the weak limit, the Lebesgue density being given by
\[
\R \ni x\mapsto \frac{1}{2p^{1/p}\Gamma(1+\frac{1}{p})}e^{-|x|^p/p},\qquad 1\leq p<\infty.
\]
This density is intimately related to $\ell_p^n$-balls, we write $\mathbb B_p^n$, and of fundamental importance in essentially any probabilistic approach to understand their geometry; see, e.g., \cite{KPT2019_I, SS1991, SZ1990} and references cited therein.  

Recently, Kim and Ramanan \cite[Theorem 2.3]{RKConLim} showed that a conditional analog to the $\ell_p$-version of the Poincar\'e--Maxwell--Borel lemma holds. Roughly speaking, they found that for $1 \leq q < p < \infty $, as the dimension $n$ of the space is high, a random point on the $\ell_p^n$-sphere, conditioned on having a sufficiently small $\ell_q$-norm, is weakly close to a random point drawn from an appropriately scaled $\ell_q^n$-sphere. As the authors point out, this means that conditioning on a sufficiently small $\ell_q$-norm induces a probabilistic change that admits a nice geometric interpretation. The strategy of proof is based on ideas from statistical mechanics and large deviations theory (see also \cite{KP2021} for more in this direction). The authors first prove a level-$2$ large deviation principle for the empirical measure of the coordinates of a random point on a properly scaled $\ell_p^n$-sphere, where the random choice is made with respect to the cone probability measure on the boundary. Fundamental to this proof  is a probabilistic representation of the cone measure in terms of the $p$-generalized Gaussian distribution \cite[Lemma 3.9]{RKConLim}, which goes back to Schechtman and Zinn \cite{SZ1990} and Rachev and R\"uschendorf \cite{RSR1991}, and allows one to go over to random vectors with independent coordinates. Let us already point out that in our generalized setting considered here, no such probabilistic representation is available. Having proved a Sanov-type large deviation principle for the empirical measure, the authors in \cite{RKConLim} move forward deducing the conditional limit theorems from this large deviation result and the famous Gibbs conditioning principle; obviously numerous (technical) hurdles need to be taken here. The use of the Gibbs conditioning principle allows one to transform a Sanov-type large deviation principle for an empirical measure into a statement about the most likely behavior of the underlying sequence of random vectors conditioned on a rare event.

In this work, we prove that an analog result holds for the general and fundamental class of Orlicz balls; Orlicz spaces are natural generalizations of $\ell_p$-spaces and intensively studied in functional analysis and optimization theory (see, e.g., \cite{HAO2006303,Kam1984,Kosmol2011Opt,Kwapien1985,LT1977,PS2012Comb,RS1988, Schuett1994}). One of the challenges that arises is that no probabilistic representation, like in the case of $\ell_p$-balls, is at our disposal. This makes an analysis more challenging. In addition, a general Orlicz function is not homogeneous, making any approach to understanding analytic, geometric or probabilistic aspects of these spaces and their unit balls even more complicated. Key to our generalization, which follows the general theme of first proving a level-$2$ large deviation principle for the sequence of empirical measures of the coordinates of random points chosen uniformly from a suitably scaled Orlicz ball and then using Gibbs conditioning, is the statistical mechanics point of view. Indeed, here conditional distributions appear in connection to the study of non-interacting particle systems (micro-canonical and canonical ensembles), where one is seeking to describe the most likely state of the system under an energy constraint (see, e.g., \cite{E1985}, \cite{LDGibsMeasRasoulAga2005}, and \cite[Subsection 1.2]{KP2021}). This point of view directly links such questions to certain Gibbs distributions. Considering such Gibbs measures at suitable critical temperatures with potentials given by an Orlicz function allows us to work around the problem arising from the lack of a probabilistic representation. This idea has recently been put forward by Kabluchko and Prochno \cite{KP2021} (see particularly Subsection 1.2 there) in their probabilistic approach to the geometry of Orlicz balls and has also successfully been used in \cite{AP2022}, \cite{barthe2021volume}, and \cite{kimliaoramanan2020asymptotic}.

\subsection{Main results}

In order to present the main results of this paper, we need to introduce some notation; for the rest we refer the reader to Section \ref{sec:notation and prelim} below.

Recall that an Orlicz function $V: \R \rightarrow [0,\infty)$ is a convex, symmetric function with $V(0) = 0$ and $V(x) > 0$ for $ x \neq 0$. For $R \in (0,\infty )$, we denote by 
	\begin{equation*}
	\mathbb{B}_R^{n,V} := \Big \{ x \in \R^n \ : \ \sum_{k=1}^n V( x_k) \leq R n  \Big \}. 
	\end{equation*} 
the associated Orlicz ball in $\R^n$. If $V: \R \rightarrow [0,\infty)$ is an Orlicz function and $ \beta \in (- \infty, 0)$, we define the corresponding Gibbs measure $\mu_{V,\beta}$ (with potential $V$ and at critical temperature $\beta$) via the Lebesgue density
\begin{equation*} 
\frac{d \mu_{V,\beta}}{dx}(x):= \frac{e^{\beta V(x)}}{\int_{\R}e^{\beta V(t)}dt}, \quad x \in \R.
\end{equation*}
We shall denote by $\mathscr{M}_1(\R)$ the set of probability measures on the Borel $\sigma$-algebra on $\R$ and, for $\mu \in \mathscr{M} _1( \R )$, we define
\begin{equation*}
m_V( \mu ) := \int_{\R} V(x)  \mu (dx) \quad \text{ and } \quad \mathscr{M}^V_1( \R ):= \Big \{ \mu \in \mathscr{M}_1( \R )  \  : \ m_V(\mu ) < \infty \Big  \}. 
\end{equation*}
In case of $V(x)=|x|^p $ for some $p \in [1,\infty)$, we write $\mathscr{M}_1^p(\R) := \mathscr{M}_1^V(\R)$.
For a metric $d_{ w} $ inducing weak convergence on $\mathscr{M}_1( \R )$, we can define a metric $d_V$ on $\mathscr{M}^V_1( \R )$ with
\begin{equation*}
d_V( \mu ,\nu ) := d_{ w} ( \mu , \nu ) + \big |   m_V( \mu ) - m_V( \nu ) \big |.
\end{equation*}

Our first main result establishes a level-$2$ or Sanov-type large deviation principle for the empirical measure of the coordinates of a vector chosen uniformly at random from an Orlicz ball.

\begin{thmalpha}
	\label{ThmLDPEmpMeasuresOrlicz}
	Let $V$ be an Orlicz function and $R\in(0,\infty)$. Assume that $ X_R^{n,V} := \big ( X_R^{n,V}(1),..., X_R^{n,V}(n) \big )\sim \U (  \mathbb{B}_R^{n,V})$. Then, there exists a unique $\overline{\alpha} \in (-\infty,0)$ such that the sequence of empirical measures
	\begin{equation}
	\label{EqEmpMeasureOrlicz}
	\mathcal{L}^{n,V} := \frac{1}{n} \sum_{i=1}^{n} \delta_{ X_R^{n,V}(i)} ,\quad{ }  n \in \N ,
	\end{equation}
	satisfies an LDP in $\mathscr{M}_1( \R )$ with the strictly convex good rate function $ \mathbb{I}_V: \mathscr{M}_1(\R) \rightarrow [0,\infty)$, where
	\begin{equation}
	\label{EqDefGRFEmpMeasOrlicz}
	\mathbb{I}_V(\mu) = \begin{cases} 
	H( \mu | \mu_{ V, \overline{\alpha} }) + \overline{\alpha} \big[ m_V( \mu ) -R  \big] & : m_V( \mu ) \leq R \\
	\infty & \text{else}. 
	\end{cases}
	\end{equation}
	In addition, if there exists another Orlicz function $W$ such that
	\begin{equation*}
	\frac{V(x)}{W(x)} \longrightarrow \infty , \quad \text{as} \quad |x| \rightarrow \infty ,
	\end{equation*}
	then, $( \mathcal{L}^{n,V} )_{n \in \N }  $ satisfies an LDP in $(\mathscr{M}^{W}_1( \R ), d_W)$.
\end{thmalpha}

\begin{rem}
	The proof of Theorem \ref{ThmLDPEmpMeasuresOrlicz} rests on a general version of the Gärtner--Ellis theorem. Beyond the Sanov-type large deviation principle on $\mathscr{M}_1( \R )$ endowed with the weak topology, also the last part of the statement, the large deviation principle on the generalized Wasserstein space $(\mathscr{M}^{W}_1( \R ), d_W)$, is essential for the proof of our second main result, Theorem \ref{ThmLimCondDist} below.
\end{rem}

Let $( X_R^{n,V} )_{n \in \N}$ be a sequence of random variables, where each $X_R^{n,V}$ is uniformly distributed on the Orlicz ball $ \OBbeta $. For fixed $k \in \N$, we denote by $ \DistkKompOrl $ the distribution of the first $k$ coordinates of $ X_R^{n,V} $, i.e.,
\begin{equation*}
\DistkKompOrl ( \cdot ) := \mathbb{P} \big[ \big ( X_R^{n, V}(1),..., X_R^{n, V}(k) \big )   \in \cdot \big] .
\end{equation*}

Let $V_1, V_2$ be two Orlicz functions and $X_1^{n,V_1 } \sim \U (  \mathbb{B}_1^{n,V_1})$. Then, for $ R \in (0, \infty )$, we define 
\begin{equation}
\label{EqCondDist}
\CondDistVoneVtwo := \mathbb{P} \big[ \big ( X_1^{n, V_1}(1),..., X_1^{n, V_1}(k) \big )   \in \cdot \big | X^{n, V_1}_1 \in \mathbb{B}^{n,V_2 }_R  \big]. 
\end{equation}

We can now state the second main result of this paper, which establishes a conditional limit theorem for random points in Orlicz balls.

\begin{thmalpha}
	\label{ThmLimCondDist}
	Let $V_1, V_2$ be Orlicz functions that satisfy
	\begin{equation}
	\label{EqCondOrliczFct}
	\frac{V_1(x)}{V_2(x)} \longrightarrow \infty  \quad{} \text{ as } \quad{} |x| \rightarrow \infty \quad \text{and} \quad  \int_{\R } V_1(x) e^{ \alpha V_2(x)} dx < \infty, 
	\end{equation}
	for all $\alpha \in (-\infty,0)$.
	Then, there exists an $ \overline{R } \in (0,\infty)$ such that for all $0 < R \leq \overline{R}$ and for fixed $k \in \N$, we have
	\begin{equation}
	\label{EqLimitMainThm}
	\lim_{ \epsilon \rightarrow 0} \lim_{n \rightarrow \infty } d_{ w} \big ( \CondDistVoneVtwoeps ,  \lambda^{( n ,k)}_{R , V_2}   \big ) = 0,
	\end{equation}
	where $d_{ w} $ denotes some metric inducing weak convergence of probability measures.
\end{thmalpha}

We shall deduce Theorem \ref{ThmLimCondDist} from our Sanov-type large deviation result in Theorem \ref{ThmLDPEmpMeasuresOrlicz}, which we combine with the so-called Gibbs conditioning principle, entropy maximization, and an Orlicz version of the Poincar\'e-Maxwell-Borel lemma. Let us continue with a few remarks related to Theorem \ref{ThmLimCondDist}.

\begin{rem}
	Under the assumptions of Theorem \ref{ThmLimCondDist}, roughly speaking, in high dimensions we have that random points in the $ V_1$-Orlicz ball, conditioned on having a small $V_2$-Orlicz radius, look like an appropriately scaled $V_2$-Orlicz ball.
\end{rem}

\begin{rem}
Theorem \ref{ThmLimCondDist} generalizes \cite[Theorem 2.3]{RKConLim}. There the authors proved a conditional limit theorem in an $\ell_p$-setting. Let us point out that there spheres rather than balls were considered.
\end{rem}

\begin{rem}
	In Corollary \ref{CorBigVtwoRadius} we will see that, if the radius $R$ in the condition of $\CondDistVoneVtwoeps$ is chosen sufficiently large, then the limiting distribution of $ \CondDistVoneVtwoeps$ is the same as the one of $ \lambda^{( n ,k)}_{1 , V_1}$. Geometrically this means that, in high dimensions, points in a $V_1$-Orlicz ball conditioned on being in a big $V_2$-Orlicz ball are uniformly distributed in the $V_1$-Orlicz ball, which is to be expected. We refer to Remark \ref{RemInterMedCaseOrliczRadius} on the phase transition in the \glqq{}intermediate case\grqq.
\end{rem} 

\begin{rem}
	The second condition in Equation \eqref{EqCondOrliczFct} ensures the existence of $m_{V_2}( \mu_{V_1, \alpha})$ for all $\alpha  \in (-\infty,0)$. This condition does not appear in case of $\ell_p$-balls, since 
	\begin{equation*}
	\int_{\R} |x|^p e^{ - |x|^q} dx < \infty 
	\end{equation*}
	for all $ 1 \leq p,q < \infty$. In particular, the second Condition in \eqref{EqCondOrliczFct} is satisfied for all polynomially bounded $V_1$. A similar condition appears in \cite[Theorem \ref{ThmLimCondDist}]{KP2021}. 
\end{rem}

\begin{rem}
	Fix $p \in [1, \infty)$ and take some Orlicz function $V$ with
	\begin{equation*}
	\frac{V(x)}{|x|^p} \longrightarrow \infty \quad \text{as} \quad |x| \rightarrow \infty .
	\end{equation*}
	Then, the sequence of empirical measures $( \mathcal{L}^{n,V})_{n \in \N}$ from Equation \eqref{EqEmpMeasureOrlicz} satisfies an LDP in the $p$-Wasserstein space $\mathscr{M}^{p}_1( \R )$ with good rate function $\mathbb{I}_V$ defined in Equation \eqref{EqDefGRFEmpMeasOrlicz} (see also Corollary \ref{CorThinshellOrl}). 
\end{rem}

\section{Notation and Preliminaries}\label{sec:notation and prelim}

\subsection{Notation}

We shall use the following standard notation. Given a Borel measurable set $A\subseteq \R^n$, we shall denote by $\vol_n$ the $n$-dimensional Lebesgue measure. If $X$ is a topological space, then for Borel sets $A \subseteq X$, we write $ A^{ \circ }$ and $\overline{A}$ for the interior and the closure of $A$ respectively.  For a sequence $( \nu_n)_{n \in \N} \subseteq \mathcal{M}_1( \R)$, which converges weakly to a distribution $ \nu \in \mathcal{M}_1( \R)$ as $n \rightarrow \infty$, we write $ \nu_n \stackrel{d}{ \longrightarrow } \nu $ as $n \rightarrow \infty$.

For a measure $ \nu \in \mathscr{M}_1( \R )$, we define the relative entropy $H( \cdot | \nu ) : \mathscr{M}_1( \R ) \rightarrow \R $ with
\begin{equation}
\label{EqDefinitionRelEntr}
H( \mu | \nu ) := \begin{cases}
\int_{\R } \log  \frac{d \mu }{d \nu } (x)  \mu (dx) & :\,  \frac{d \mu }{d \nu } \text{ exists} \\
\infty &\, \text{ else}.
\end{cases}
\end{equation}
A related quantity is the entropy $h : \mathscr{M}_1( \R ) \rightarrow \R $ of a measure, which is defined as 
\begin{equation*}
h( \nu ) := \begin{cases}
- \int_{\R} \log \frac{d \nu }{ dx } ( x) d x & :\, \frac{d \nu }{ dx } \text{ exists} \\
\infty & \,\text{ else}.
\end{cases}
\end{equation*}

For $ \mu \in \mathscr{M}_1( \R)$ and a measurable function $f : \R \rightarrow \R$, we set
\begin{equation*}
\mathbb{E}_{\mu } [f] := \int_{ \R} f(x) \mu(dx),
\end{equation*}
provided the latter exists. Also, for $\mu \in \mathscr{M}_1( \R)$, we define $\delta_{\mu}:  \mathscr{M}_1( \R) \rightarrow \R $ with 
\begin{equation*}
\delta_{ \mu }( \nu ) = \begin{cases}
1 & : \nu = \mu \\
0 & \text{else} .
\end{cases}, \qquad{}\nu \in \mathscr{M}_1( \R).
\end{equation*}
For $x \in \R$, $\delta_x$ denotes the Dirac measure.

We will frequently work with the space of bounded and continuous functions mapping from $\R$ to $\R$, which we shall denote by $C_b(\R)$.

\subsection{Orlicz functions \& Gibbs measures}

Let us recall here what an Orlicz function is and introduce Gibbs measures whose potentials are given by Orlicz functions.

\begin{df}
		An Orlicz function $V: \R \rightarrow [0,\infty)$ is a convex, symmetric function with $V(0) = 0$ and $V(x) > 0$ for $ x \neq 0$. We define the associated Orlicz ball 
		\begin{equation*}
		\mathbb{B}_R^{n,V} := \Big \{ x \in \R^n \ : \ \sum_{k=1}^n V( x_k) \leq R n  \Big \} 
		\end{equation*}
		for some $ R > 0$ and in case of $R=1$, we write $ \mathbb{B}^{n,V} : =  \mathbb{B}_1^{n,V}$. 
\end{df}

For some Orlicz function $V: \R \rightarrow [0,\infty)$ and $ \beta \in (- \infty, 0)$, we define the Gibbs measure $\mu_{V,\beta}$ via the Lebesgue density
\begin{equation*}
\frac{d \mu_{V,\beta}}{dx}(x):= \frac{e^{\beta V(x)} }{\int_{\R} e^{\beta V(t)} dt}, \quad x \in \R.
\end{equation*}

\subsection{Basics from large deviation theory and probability}

Consider a sequence of random variables $( \xi_n)_{n \in \N}$ taking values in a topological space $X$, equipped with a Borel sigma algebra and a probability measure $\mathbb{P}$ on it. We say that $( \xi_n)_{ n \in \N}$ satisfies an LDP in $X$, if there exists a good rate function (GRF) $\I : X \rightarrow [0, \infty ]$, i.e $\I $ has compact level sets, such that
\begin{equation}
\label{EqIntrLDP}
- \inf_{ x \in A^{ \circ }} \I( x) \leq \liminf_{ n \rightarrow \infty } \frac{1}{n} \log \prb[ \xi_n \in A^{ \circ } ] \leq \limsup_{ n \rightarrow \infty } \frac{1}{n} \log \prb[ \xi_n \in \overline{A}] \leq - \inf_{ x \in  \overline{A}} \I( x) 
\end{equation} 
for all Borel sets $A \subseteq X$.

There are different ways to show that a sequence of random variables satisfies an LDP, the most commonly used is the so called contraction principle (see, e.g., Theorem 4.2.1 in \cite{DZ2011}).

\begin{lem}
\label{LemContractionPrinciple}
Let $X$ and $Y$ be Hausdorff topological spaces and $f: X \rightarrow Y$ be a continuous function. Let $( \xi_n)_{ n \in \N}$ be a sequence of random variables that satisfies an LDP in $X$ with GRF $\mathbb{I}: X \rightarrow [0, \infty ]$. Then, the sequence $( f( \xi_n))_{n \in \N}$ satisfies an LDP in $Y$ with the GRF $\mathbb{I}' : Y \rightarrow [0, \infty ]$, where
\begin{equation*}
\mathbb{I}'( y) := \inf_{ x \in f^{-1}( \{ y \})} \mathbb{I}(x).
\end{equation*}
\end{lem}

When working with LDPs it may occur that one encounters different topologies on the same ambient space $X$. Corollary \ref{CorCoarserFinerTop} is a useful tool when dealing with that situation.

\begin{df}
Let $( \xi_n)_{n \in \N}$ be a sequence of random variables taking values in a topological space $X$. The sequence $( \xi_n)_{n \in \N}$ is called exponentially tight if for all $ \alpha \in (-\infty, 0)$ there exists a compact set $K_{ \alpha } \subseteq X$ such that
\begin{equation*}
\limsup_{n \rightarrow \infty }\frac{1}{n} \log \prb \big [  \xi_n \in K_{ \alpha }^c \big ] < \alpha . 
\end{equation*}
\end{df}
\begin{cor}(Corollary 4.2.6 in \cite{DZ2011})
\label{CorCoarserFinerTop}
Let $(\xi_n)_{n \in \N}$ be an exponentially tight sequence of random variables on a topological space $X$ equipped with a topology $\tau_1$. If $(\xi_n)_{n \in \N}$ satisfies an LDP with respect to a Hausdorff topology $\tau_2$ that is coarser than $ \tau_1$, then the same LDP holds with respect to the finer topology $\tau_1$.
\end{cor} 

Lastly, we present a version of the Gärtner--Ellis Theorem, which will play an important role in the proof of Theorem \ref{ThmLDPEmpMeasuresOrlicz} from Section 2.

\begin{lem}(Corollary 4.16.14 in \cite{DZ2011})
\label{LemGaertnerEllis}
Let $( \xi_n)_{n \in \N}$ be a sequence of exponentially tight random variables on the locally convex Hausdorff topological vector space $\mathcal{E}$. Suppose the Gärtner--Ellis limit
\begin{equation*}
\Lambda( \lambda ):= \lim_{n \rightarrow \infty } \frac{1}{n} \log \mathbb{E} \Big [ e^{ n  \lambda ( \xi_n ) } \Big ], \quad \lambda \in \mathcal{E}^{*},
\end{equation*}
exists in $\R $ and is Gateaux-differentiable, where $\mathcal{E}^{*}$ is the topological dual of $\mathcal{E}$. Then, $( \xi_n)_{n \in \N}$ satisfies the LDP in $\mathcal{E}$ with the convex GRF $\Lambda^{*}: \mathcal{E} \rightarrow [0, \infty]$, given by
\begin{equation*}
\Lambda^{*}( x) := \sup_{ \lambda \in \mathcal{E}^{*} } \big \{  \lambda(x) - \Lambda( \lambda )  \big \}.
\end{equation*}
\end{lem}

For a measure $\nu \in \mathscr{M}_1( \R)$ we recall the relative entropy $H( \cdot | \nu ) : \mathscr{M}_1( \R) \rightarrow \R$ from Equation \eqref{EqDefinitionRelEntr} in Section 2.1, for which the following variational formula holds.

\begin{proposition}(Theorem 5.6 in \cite{LDGibsMeasRasoulAga2005})
\label{PropVarFormulaRelEntr}
Let $\mu, \nu \in \mathscr{M}_1( \R)$. Then
\begin{equation*}
H( \mu | \nu ) = \sup_{ \lambda \in C_b( \R )} \big \{  \mathbb{E}_{ \mu } [ \lambda ] - \log \mathbb{E}_{ \nu } [ e^{\lambda }]\big \}.
\end{equation*}
\end{proposition}

When working with the space of probability measures $\mathscr{M}_1( \R )$, the notion of tightness plays a major role.
\begin{df}
	A subset $L$ of $ \mathscr{M}_1(\R)$ is called tight if, for all $\eps > 0$, there exists a compact set $K_{ \eps } \subseteq \R$ such that 
	\begin{equation*}
	\sup_{ \mu \in L} \mu( K_{ \eps }) \geq 1 - \eps .
	\end{equation*}
\end{df}

The following result is known as Gibbs conditioning principle, one of the fundamental tools in the theory of large deviations. It will be used in the proof of Proposition \ref{PropMainRes}, which is itself used in the proof of Theorem \ref{ThmLimCondDist}.

\begin{proposition}(Gibbs conditioning principle, Theorem 7.1 in \cite{LeonardEntrProj2010})
	\label{PropGibbsCond}
	Let $X$ be a topological space, and let $( \xi_n)_{n \in \N}$ be a sequence of $X$-valued random variables that satisfies an LDP in $X$ with GRF $\I$. In addition, let $F \subseteq X$ be a subset that satisfies
	\begin{itemize}
		\item[] (1) $ \I(F) := \inf_{ x \in F} \I(x) < \infty  $;
		\item[] (2) $ F$ is closed;
		\item[] (3) $F = \bigcap_{ \eps  > 0}  F_{ \eps }$ for a family of sets $( F_{ \eps })_{ \eps > 0} $ such that $F_{\eps} \subseteq X $ for $\eps > 0$ and $\prb [ \xi_n  \in F_{ \eps } ]> 0$ for all $ \eps > 0 $ and $n \in \N $;
		\item[] (4) $F \subseteq ( F_{ \eps } )^\circ$ for all $\eps > 0$.
	\end{itemize}
	Let $\mathcal{M}_F$ be the set of $x \in F$ which minimize $\I $. That is,
	\begin{equation*}
	\mathcal{M}_F := \{ x \in F : \I(x) = \I(F )  \}.
	\end{equation*}
	Then, for all open $G \subseteq X$ such that $\mathcal{M}_F \subseteq G$, we have
	\begin{equation*}
	\limsup_{ \eps \rightarrow 0 } \limsup_{ n \rightarrow \infty } \frac{1}{n} \log \prb [ \xi_n \notin G \ | \ \xi_n \in F_{ \eps } ] < 0. 
	\end{equation*} 
	As a consequence, if $\mathcal{M}_F = \{ \overline{x} \}$ is a singleton, then 
	\begin{equation}
	\label{EqGibbsCondPrSingl}
	\lim_{ \eps \rightarrow 0} \lim_{n \rightarrow \infty } \prb[ \xi_n \in \cdot | \xi_n \in F_{ \eps } ] = \delta_{ \overline{x} }.
	\end{equation}
	
\end{proposition}

In the proof of Proposition \ref{PropMainRes} we shall make use of Proposition \ref{propPropchaos}, which relies on \cite[Proposition 2.1]{PropOfChaos}. We start with a definition, where we adapt the notation from \cite{PropOfChaos} in order to stay consistent with the notation used in this article.

\begin{df}
\label{Defmuchaotic}
Let $( \mu_{\eps }^n)_{n \in \N, \eps > 0}$ be a family of probability measures, where for each $\eps > 0$ and each $n \in \N$, $\mu_{\eps }^n$ is a symmetric distribution on $\R^n$. We say that $ ( \mu_{\eps }^n)_{n \in \N, \eps > 0}$ is $\mu$-chaotic if there exists $\mu \in \mathscr{M}_1( \R )$ such that, for fixed $k \in \N$ and $ \phi_1,...,\phi_k \in C_b( \R)$, we have that
\begin{equation}
\label{EqMuChaoticlimit}
\lim_{ \eps \rightarrow 0} \lim_{n \rightarrow \infty} \int_{ \R^k} \prod_{i=1}^{k} \phi_i(x_i)  \mu_{\eps }^{n,k}(dx) = \prod_{i=1}^{k} \int_{ \R} \phi_i(x)  \mu ( dx),
\end{equation}
where
$\mu_{ \eps }^{n,k}$ denotes the marginal distribution of the $k$ first coordinates of $\mu_{\eps }^n$.
\end{df}

\begin{rem}
\label{RemWeakconvergence}
In particular, the limit in Equation \eqref{EqMuChaoticlimit} implies weak convergence of the family $( \mu^{n,k}_{ \eps })_{n \in \N, \eps > 0}$ towards $ \mu^{\otimes k}$, where the latter denotes the $k$-fold product measure of $\mu$.
\end{rem}

We have the following characterization of the notion defined in Definition \ref{Defmuchaotic}.

\begin{proposition}
\label{propPropchaos}
Let $( \mu^{n,k}_{ \eps })_{n \in \N, \eps > 0}$ be a family of symmetric probability measures, where for each $\eps > 0$ and $n \in \N$, $\mu_{ \eps }^n$ is a symmetric distribution on $\R^n$. For all $\eps > 0$ and $n \in \N$, let $X_{ \eps }^n = (X_{\eps }^n(1),...,X_{ \eps }^n(n)) \in \R^n$ be distributed according to $\mu_{ \eps }^n$. Then, the family $( \mu^{n,k}_{ \eps })_{n \in \N, \eps > 0}$ is $\mu$-chaotic if and only if there exists $\mu \in \mathscr{M}_1(\R)$ such that
\begin{equation}
\label{EqLimitEmpMeasuresMuChaotic}
\frac{1}{n} \sum_{ i=1}^n \delta_{ X_{ \eps }^{n}(i)} \stackrel{d}{\longrightarrow} \mu ,
\end{equation} 
as $n \rightarrow \infty$ followed by $\eps \rightarrow 0$. 
\end{proposition}
\begin{proof}
Similar to the proof of \cite[Proposition 2.2]{PropOfChaos}.
\end{proof}

\subsection{The generalized Wasserstein spaces} 

Let $ \mathscr{M} _1( \R )$ be the set of probability measures on the Borel $\sigma$-algebra on $\R $ and $V$ be an Orlicz function and recall the following quantities from Section 2. 

\begin{df}
	For $\mu \in \mathscr{M} _1( \R )$, we define
	\begin{equation}
	\label{DefGenWassSpace} 
	m_V( \mu ) := \int_{\R} V(x)  \mu (dx) \quad \text{ and } \quad \mathscr{M}^V_1( \R ):= \Big \{ \mu \in \mathscr{M}_1( \R )  \  : \ m_V(\mu ) < \infty \Big  \}. 
	\end{equation}
	In case of $V(x)=|x|^p $ for some $p \in [1,\infty)$, we write $m_p := m_V$ and $\mathscr{M}_1^p(\R) := \mathscr{M}_1^V(\R)$.
	For a metric $d_{ w} $ inducing weak convergence on $\mathscr{M}_1( \R )$, we can define a metric $d_V$ on $\mathscr{M}^V_1( \R )$ with
	\begin{equation}
	\label{DefGenWassMetric}
	d_V( \mu ,\nu ) := d_{ w} ( \mu , \nu ) + \big |   m_V( \mu ) - m_V( \nu ) \big |.
	\end{equation}
\end{df}

\begin{rem}
	\label{RemContinuityVthmomentmap}
	By the very definition of the metric $d_V$ on $\mathscr{M}^V_1$, the mapping $m_V : \mathscr{M}^V_1( \R ) \rightarrow \R $ is a continuous one.
\end{rem}
\begin{rem}
	In case of $V(x) = |x|^p$, for some $ p \in [1, \infty)$, the space $(\mathscr{M}^V_1( \R ), d_V)$ becomes the $p$-Wasserstein space, which is complete and separable (see, e.g, Chapter 6 in \cite{OptTransport2009}).
\end{rem}

Note that for general Orlicz functions $V$, it is not claimed that $(\mathscr{M}^V_1( \R ), d_V)$ is Polish, but we have the following result which suffices for our purposes.

\begin{proposition}(Lemma C.1 in \cite{DLRLDGibbsMeas2020})
	\label{PropGenWassSpace}
	Let $V$ be an Orlicz function and $\mathscr{M}^V_1( \R ) $ and $d_V$ be defined as in \eqref{DefGenWassSpace} and \eqref{DefGenWassMetric}, respectively. Then, for a sequence $( \mu_n)_{ n \in \N} \subseteq \mathscr{M}^V_1( \R )$ and some $ \mu \in \mathscr{M}^V_1( \R )$, we have $d_V( \mu_n, \mu ) \rightarrow 0$ as $n \rightarrow \infty$ if and only if
	\begin{equation}
	\label{EqCondGenWassConvergence}
	d_{ w}( \mu_n, \mu ) \rightarrow 0 \quad \text{and} \quad \lim_{ r \rightarrow \infty} \sup_{n \in \N } \int_{|x| > r} V(x)  \mu_n(dx) \rightarrow 0 \quad \text{as} \quad  n \rightarrow \infty.  
	\end{equation}
	Furthermore, the metric space $(\mathscr{M}^V_1( \R ), d_V)$ is separable.
\end{proposition}

\begin{proposition}
	\label{PropCompactSetpWasserst}
	Let $M \in (0, \infty )$, then the set $ K_1^M := \{ \nu \in \mathscr{M}_1(\R) : \int_{ \R} |x| d \nu( x)  \leq M \}$ is compact in the weak topology on $\mathscr{M}_1( \R )$.
\end{proposition}
\begin{proof}
	The proof is analogue to the proof of Lemma 3.14 in \cite{RKConLim}.
\end{proof}

\begin{proposition}
	\label{PropCompactSetGenWasserst}
	Let $V_1, V_2$ be two Orlicz functions with
	\begin{equation*}
	\frac{V_1(x)}{V_2(x) } \longrightarrow \infty \quad \text{as} \quad |x| \rightarrow \infty .
	\end{equation*}
	Then, for $M \in (0,\infty)$, the set $K_{V_1}^M := \big \{ \nu \in \mathscr{M}_1( \R ) : m_{V_1}( \nu  ) \leq M \big \}$ is compact in the $d_{V_2}$-topology on $\mathscr{M}^{V_2}_1( \R ) $.
\end{proposition}

\begin{proof}
	Since the function $V_1 : \R \rightarrow [0,\infty)$ is convex, it has compact level sets. Thus, by Lemma 3.2 in \cite{DLRLDGibbsMeas2020}, $K_{V_1}^M$ is tight. By Prohorov's Theorem (see, e.g., Theorem D.9 in \cite{DZ2011}), $K_{V_1}^M$ is precompact in the weak topology. To see that $K_{V_1}^M$ is also closed in the weak topology, we take a sequence $(\mu_n)_{n \in \N} \subseteq K_{V_1}^M$ with $ \mu_n \rightarrow \mu \in \mathscr{M}_1( \R )$ weakly. Then, for fixed $C \in (0,\infty)$, we get
	\begin{align*}
	M & \geq \int_{\R } V_1(x)  \mu_n(dx) \\
	& \geq \int_{\R } \max \{  C, V_1(x) \}  \mu_n(dx) \longrightarrow \int_{\R } \max \{  C, V_1(x) \} \mu(dx) \quad \text{as} \quad n \rightarrow \infty,
	\end{align*}
	where the convergence in the second line holds since $ \max \{ C, V_1 \}$ is continuous and bounded. By monotone convergence, we get
	\begin{equation*}
	M \geq \lim_{ C \rightarrow \infty } \int_{\R } \max \{  C, V_1(x) \} \mu(dx) = \int_{\R }  V_1(x)   \mu(dx).
	\end{equation*}
	Therefore, $\mu \in K_{V_1}^M$, which shows weak compactness of $K_{V_1}^M$. Now take a sequence $( \mu_n )_{n \in \N} \subseteq K_{V_1}^M$. Then, by weak compactness, there exists a subsequence $( \mu_{n_k })_{ k \in \N}$ and $ \mu \in K_{V_1}^M$ such that $\mu_{n_k} \rightarrow \mu $ weakly. We claim that $d_{V_2}( \mu_{n_k} , \mu ) \rightarrow 0$ as $k \rightarrow \infty $ (see Equation \eqref{DefGenWassMetric} for the definition of this metric). In order to show that, we work with the characterization of convergence in $\mathscr{M}^{V_2}_1( \R )$ given in Proposition \ref{PropGenWassSpace}. We get
	\begin{align*}
	\lim_{r \rightarrow \infty } \sup_{k \in \N} \Big \{ \int_{ |x| > r } V_2(x)  \mu_{ n_k}(dx)   \Big \} & = \lim_{r \rightarrow \infty } \sup_{k \in \N} \Big \{ \int_{ |x| > r } \frac{V_2(x)}{V_1(x)} V_1(x)   \mu_{ n_k}(dx)   \Big \} \\
	& \leq \lim_{r \rightarrow \infty } \sup_{ |x| \geq r} \frac{V_2(x)}{V_1(x)} \sup_{k \in \N} \Big \{ \int_{ |x| > r }  V_1(x)  \mu_{ n_k}(dx)   \Big \} \\
	& \leq M \limsup_{ x \rightarrow \infty} \frac{V_2(x)}{V_1(x)} = 0,
	\end{align*}
	where we used that $ ( \mu_{n_k})_{ k \in \N} \subseteq K_{V_1}^M$ and $ \lim_{ |x| \rightarrow \infty } V_2(x) / V_1(x) = 0$. Thus, we have shown that every sequence in $K_{V_1}^M$ has a $d_{V_2}$ convergent subsequence, which implies that $K_{V_1}^M$ is compact in $\mathscr{M}^{V_2}_1( \R )$. 
\end{proof}

\section{Proof of Theorem \ref{ThmLDPEmpMeasuresOrlicz}}

We shall now continue and present the proof of the Sanov-type large deviation principle presented in Theorem \ref{ThmLDPEmpMeasuresOrlicz}.

\begin{lem}
	\label{AuxLemmaExpTilt}
	Let $V $ be an Orlicz function and let $R \in (0,\infty)$ . Define $\varphi_V: (-\infty,0) \times C_b(\R) \rightarrow \R $ as
	\begin{equation}
	\varphi_V( \alpha, \lambda ) : = \log \int_{\R } e^{ \alpha V(x) + \lambda(x)} dx .
	\end{equation}
	Then, for each $\lambda \in C_b(\R )$, there exists an unique $ \alpha= \alpha( \lambda ) \in (-\infty,0)$ such that 
	\begin{equation*}
	\frac{d }{d \alpha} \varphi_V( \alpha( \lambda ), \lambda ) = R \quad \text{and} \quad  \frac{d^2 }{d \alpha^2} \varphi_V( \alpha( \lambda ), \lambda ) \in ( 0, \infty ).
	\end{equation*}
	Moreover, the mapping $\lambda \mapsto \varphi_V( \alpha( \lambda ), \lambda) $ is Fréchet-differentiable and we have that
	\begin{equation}
	\label{EqMinAlphaLam}
	\varphi_V( \alpha( \lambda ), \lambda) - \alpha( \lambda) R = \inf_{ \alpha < 0} \big[ \varphi_V( \alpha, \lambda ) - \alpha R \big].
	\end{equation}
\end{lem}

\begin{rem}
To be precise, the function $\alpha( \lambda )$ also depends on $R$ and on $V$, but to keep the notational burden low, we omit these variables in this section.
\end{rem}
\begin{proof}(of Lemma \ref{AuxLemmaExpTilt})
	Fixing $\lambda \in C_b( \R )$, one can check easily that 
	\begin{equation}
	\label{EqPhiAlphaLambdaEst}
	e^{ - 2 || \lambda ||_{ \infty }} \frac{d }{d \alpha} \varphi_V( \alpha, 0 ) \leq  \frac{d }{d \alpha}  \varphi_V( \alpha, \lambda  ) \leq e^{ 2 || \lambda ||_{ \infty }} \frac{d }{d \alpha} \varphi_V( \alpha, 0 ).
	\end{equation}
		We have that $ \frac{d }{d \alpha} \varphi_V( \alpha, 0 ) \rightarrow 0 $ as $\alpha \rightarrow - \infty$ as well as $ \frac{d }{d \alpha} \varphi_V( \alpha, 0 ) \rightarrow \infty $ as $\alpha \rightarrow 0$. We show this. First, we observe that
		\begin{equation*}
		\alpha \mapsto \frac{d }{d \alpha} \varphi_V( \alpha, 0 ) = \int_{ \R  } V(x) e^{ \alpha V(x)- \varphi_V( \alpha, 0) }dx 
		\end{equation*}
		is, as a parameter integral, continuous on the negative real numbers. Now, for $M \in (0,\infty) $,
		\begin{align*}
		\int_{ \R  } V(x) e^{ \alpha V(x) - \varphi_V( \alpha, 0)} dx & \geq 2 e^{- \varphi_V( \alpha, 0)} \int_{ M}^{ \infty }  V( x) e^{ \alpha V(x)} d x \\
		&  \geq 2 V(M) e^{- \varphi_V( \alpha, 0)} \int_{ M}^{ \infty }  e^{ \alpha V(x)} d x \\
		& \geq  V(M) \Big( 1 - 2 e^{- \varphi_V( \alpha, 0)} \int_{ 0}^{ M } e^{ \alpha V(x)} dx \Big)  \xrightarrow{ \alpha \rightarrow 0-} V(M ).
		\end{align*}
		The limit in the last line holds, since $ \varphi_V( \alpha, 0) \longrightarrow \infty $ as $ \alpha \rightarrow 0$ and $ \int_{ 0}^{ M } e^{ \alpha V(x)} dx  \leq M $ for all $ \alpha \in (-\infty, 0)$. Thus, we get
		\begin{equation*}
		\liminf_{ \alpha \rightarrow 0 } \int_{ \R  } V(x) e^{ \alpha V(x) - \varphi_V( \alpha, 0) }dx \geq V(M) \longrightarrow \infty \quad{} \text{as} \quad{ } M \rightarrow \infty. 
		\end{equation*}
		Next, we prove the other limit. In contrast, for $ \alpha \rightarrow -	\infty$, we fix some $ x \in (0,\infty)$, where we have
		\begin{align*}
		e^{\alpha V(x)- \varphi_V( \alpha, 0)} & = \frac{e^{\alpha V(x)}}{ \int_{ \R } e^{ \alpha V(t) } dt } \\
		& \leq \frac{e^{\alpha V(x)}}{ \int_{ -x/2 }^{x/2} e^{ \alpha V(t) } dt } \\
		& \leq \frac{1}{x} e^{ \alpha( V(x)- V( x/2) ) } \xrightarrow{ \alpha \rightarrow -\infty } 0.
		\end{align*}
		In the last line, we used that $ V(x) - V( x/2) > 0$, which holds since $V$ is a convex function. Hence, $ \int_{ \R } V(x) e^{ \alpha V(x) - \varphi_V( \alpha, 0) }dx  \longrightarrow 0$ as $\alpha  \rightarrow - \infty $. 
	Thus, by Equation \eqref{EqPhiAlphaLambdaEst}, we have that both 
	\begin{equation*}
	\frac{d }{d \alpha} \varphi_V( \alpha, \lambda  ) \xrightarrow{ \alpha \rightarrow - \infty} 0 \qquad \text{and} \qquad \frac{d }{d \alpha} \varphi_V( \alpha, \lambda  ) \xrightarrow{ \alpha \rightarrow 0} \infty. 
	\end{equation*}
	This ensures that we find $ \widetilde{ \alpha } \in (-\infty,0)$ with $  \frac{d }{d \alpha} \varphi_V( \widetilde{ \alpha } , \lambda  ) = R $. Now consider a random variable $Z$ distributed according to the Gibbs density $ e^{ \widetilde{\alpha } V(x) + \lambda(x) - \varphi_V( \widetilde{\alpha }, \lambda) }$. Then, we have that 
	\begin{equation}
	\label{EqExpValueVofZ}
	\mathbb{E}[V(Z)] =  \int_{ \R} V(x) e^{ \widetilde{ \alpha } V(x) + \lambda(x)- \varphi_V( \widetilde{\alpha }, \lambda )} dx = \frac{d }{d \alpha} \varphi_V( \widetilde{ \alpha } , \lambda  ) = R
	\end{equation}
	 and 
	 \begin{gather}
	 \begin{split}
	 \label{EqVarianceofVofZ}
	 \mathbb{V} [ V(Z)] & = \int_{ \R} V(x)^2 e^{ \widetilde{ \alpha }V(x) + \lambda(x) - \varphi_V( \widetilde{ \alpha }, \lambda)} dx - \Big( \int_{ \R} V(x) e^{ \widetilde{ \alpha } V(x) + \lambda(x) - \varphi_V( \widetilde{ \alpha }, \lambda)}dx  \Big)^2\\
	 & =\frac{d ^2}{d \alpha^2} \varphi_V( \widetilde{ \alpha } , \lambda  ).
	 \end{split}
	 \end{gather}
	 Hence, $ \frac{d ^2}{d \alpha^2} \varphi_V( \widetilde{ \alpha } , \lambda  ) \in (0,\infty) $. We now show that $ \widetilde{\alpha}$ is unique. The implicit function theorem (see, e.g., Theorem 5.9 in \cite{FundDiffGeoLang1999}) applied to the mapping 
	$ \alpha \mapsto \frac{d}{d \alpha } \varphi_V (  \alpha  , \lambda )  $ $-R$
	guarantees that $ \widetilde{ \alpha } =: \alpha( \lambda)$ is unique and Fréchet-differentiable. Thus, as a composition of Fréchet-differentiable mappings, $ \varphi_V( \alpha( \lambda ) , \lambda )$ is Fréchet-differentiable. To see that $ \alpha( \lambda )$ is the minimizer in \eqref{EqMinAlphaLam}, we note that by the Cauchy-Schwarz inequality, $\alpha \mapsto \varphi_V( \alpha, \lambda )$ is strictly convex and so is $\alpha \mapsto \varphi_V( \alpha, \lambda ) - \alpha R$. Since $ \varphi_V( \alpha, \lambda ) \rightarrow \infty$ as $\alpha \rightarrow 0$, we find an unique minimizer $\overline{\alpha}( \lambda ) \in (-\infty,0)$ with 
	\begin{equation*}
	\frac{d}{d \alpha } \varphi_V ( \overline{ \alpha}( \lambda ) , \lambda ) - R = 0
	\end{equation*}
	and hence $ \overline{ \alpha }( \lambda ) = \alpha( \lambda )$.
	
\end{proof} 

In the proof Theorem \ref{ThmLDPEmpMeasuresOrlicz}, the Legendre transform of a certain functional containing $ \varphi_V( \alpha, \lambda ) $ from Lemma \ref{AuxLemmaExpTilt} will appear, which we shall compute in the following proposition.

\begin{proposition}
	\label{PropLegTransformOrlicz}
	Let $V : \R \rightarrow [0,\infty) $ be an Orlicz function and let $R \in (0,\infty) $. Define $ \Lambda_V : C_b( \R) \rightarrow \R $ as
	\begin{equation}
	\label{EqCumGenFctOrl}
	\Lambda_V( \lambda ) := \varphi_V( \alpha( \lambda ), \lambda ) - \alpha( \lambda ) R - \varphi_V( \alpha( 0) , 0) + \alpha( 0) R,
	\end{equation}
	where $\varphi_V( \alpha, \lambda )$ and $\alpha( \lambda )$ are defined as in Lemma \ref{AuxLemmaExpTilt}. Then, the Legendre transform $\mathbb{I}_V: \mathscr{M}_1( \R) \rightarrow [0,\infty) $ of $ \Lambda_V $ is given as
	\begin{equation}
	\label{EqGRFLDPempMOrl}
	\mathbb{I}_V(\mu) := \begin{cases} 
	H( \mu | \mu_{ V, \alpha(0) }) + \alpha(0) \big[ m_V( \mu ) -R  \big] & : m_V( \mu ) \leq R \\
	\infty & \text{else}, 
	\end{cases}
	\end{equation}
	where for $\beta \in (-\infty,0)$, $\mu_{ V, \beta}$ is the distribution with Lebesgue density 
	\begin{equation*}
	\frac{d \mu_{V, \beta } (x)}{dx }= e^{ \beta V(x) - \varphi_V( \beta , 0)} , \quad x \in \R.
	\end{equation*}
    In addition, $\mathbb{I}_V$ is strictly convex where finite.
\end{proposition}

\begin{rem}
	\label{RemGRFOrl} 
	If $ \mu \in \mathscr{M}_1( \R )$ with $m_V( \mu ) \leq R$ and $\mu $ is absolutely continuous with respect to Lebesgue measure, then we have
	\begin{equation*}
	\mathbb{I}_V ( \mu ) = -h( \mu ) + \varphi_V( \alpha( 0), 0) - \alpha( 0) R ,
	\end{equation*}
	where $ h( \mu) $ is the entropy of $\mu $. Indeed, for $\mu \in \mathscr{M}^V_1(\R)$, we have
\begin{align*}
H( \mu | \mu_{V, \alpha(0)})  &= \int_{ \R } \log \frac{d \mu }{ d \mu_{ V, \alpha(0) }}(x)  \mu( dx) \\
& = \int_{ \R } \log \frac{d \mu }{ d x}(x)  \mu( dx)
+ \int_{ \R } \log \frac{d x }{ d \mu_{ V, \alpha(0) }}(x)  \mu( dx) \\
& = -h(\mu ) - \alpha(0) m_V( \mu ) - \varphi_V( \alpha(0), 0) ,
\end{align*}
where we used that $ \frac{d x}{d \mu_{ V, \alpha(0 ) }}(x) = e^{ -\alpha(0) V(x) + \varphi_V( \alpha(0) , 0)}$. For the GRF $\I_V( \mu)$, we thus get
\begin{align*}
\I_V( \mu ) & = H( \mu | \mu_{ V, \alpha(0) }) + \alpha(0) \big[  m_V(\mu) -R  \big] \\
& = -h(\mu ) - \alpha(0) m_V( \mu ) - \varphi_V( \alpha(0), 0) + \alpha(0) \big[  m_V(\mu) -R  \big] \\
& =  -h( \mu ) + \varphi_V( \alpha( 0), 0) - \alpha( 0) R.
\end{align*}
\end{rem}

\begin{proof}(of Proposition \ref{PropLegTransformOrlicz})	
	By definition, the Legendre transform $\mathbb{I}_V$ of $\Lambda_V $ is given as
	\begin{align*}
	\mathbb{I}_V(\mu ) : &= \sup_{ \lambda \in C_b( \R )} \big[  \mathbb{E}_{ \mu } [ \lambda ] - \varphi_V( \alpha( \lambda ), \lambda ) + \alpha( \lambda ) R   \big] + \varphi_V( \alpha( 0), 0) - \alpha(0) R  \\
	& = \sup_{ \lambda \in C_b( \R )} \big[  \mathbb{E}_{ \mu } [ \lambda ] +  \sup_{ \beta < 0} \big\{   \beta R - \varphi_V( \beta, \lambda )  \big \}   \big] + \varphi_V( \alpha( 0), 0) - \alpha(0) R,
	\end{align*}
	where the second equality uses the fact that $\alpha( \lambda)$ minimizes $\beta \mapsto \varphi_V( \beta, \lambda) - \beta R$ (see Lemma \ref{AuxLemmaExpTilt}). By rearranging the terms and interchanging the suprema, we get
	\begin{align}
	\nonumber
	\mathbb{I}_V(\mu ) & =  \sup_{ \beta < 0}  \big[  \beta R  +  \sup_{ \lambda \in C_b( \R )} \big\{ \mathbb{E}_{ \mu } [ \lambda ] - \varphi_V(  \beta, \lambda )  \big  \}   \big] + \varphi_V( \alpha( 0), 0) - \alpha(0) R \\
	\label{EqGRFrelEntr}
	& = \sup_{ \beta < 0}  \big[  \beta R  +  H( \mu | \mu_{V, \beta })  - \varphi_V( \beta , 0)  \big] + \varphi_V( \alpha( 0), 0) - \alpha(0) R .
	\end{align}
	Here we used the variational formula for the relative entropy $H$ given in Proposition \ref{PropVarFormulaRelEntr}.
	Next, we take a closer look at the relative entropy $H( \mu | \mu_{V,  \beta })$ in order to obtain the desired representation of $\mathbb{I}_V$. First, we note that 
	\begin{equation*}
	\frac{d \mu_{ V, \alpha(0)}}{ d \mu_{V, \beta }}(x) = \exp \big( \varphi_V( \beta, 0) - \varphi_V( \alpha( 0), 0) - \beta V(x) + \alpha( 0) V(x)  \big), \quad x  \in \R.
	\end{equation*}
	This leads to
	\begin{align*}
	H( \mu | \mu_{\beta }) & = \int_{\R} \log \Big( \frac{d \mu }{ d \mu_{ V, \beta }}(x) \Big ) \mu( dx) \\
	& = \int_{\R} \log \Big( \frac{d \mu }{ d \mu_{ V, \alpha(0) }}(x) \Big ) \mu( dx) + \int_{\R} \log \Big( \frac{d \mu_{ V, \alpha(0)} }{ d \mu_{ V, \beta} }(x) \Big )  \mu( dx) \\ 
	&= H( \mu | \mu_{ V, \alpha(0)}) + \varphi_V( \beta , 0) - \varphi_V( \alpha( 0), 0) + ( \alpha(0) - \beta ) m_V( \mu ).
	\end{align*}
	From Equation \eqref{EqGRFrelEntr}, it follows that
	\begin{align*}
	\mathbb{I}_V( \mu ) & = \sup_{ \beta < 0}  \big[  \beta R  +  H( \mu | \mu_{ V, \beta})  - \varphi_V( \beta , 0)  \big] + \varphi_V( \alpha( 0), 0) - \alpha(0) R \\
	&= 
	H( \mu | \mu_{ V, \alpha(0) }) + \sup_{ \beta < 0}  \big[  \beta R   -  \beta m_V( \mu ) \big] + \alpha( 0) ( m_V( \mu ) - R ) \\
	& = \begin{cases} 
	H( \mu | \mu_{ V, \alpha(0) }) + \alpha( 0) ( m_V( \mu ) - R ) &:  m_V( \mu) \leq R \\
	\infty & \text{else}.
	\end{cases}
	\end{align*}
	For the strict convexity, we just note that $\mu \mapsto H( \mu | \mu_{ V, \alpha(0)})$ is strictly convex (see, e.g., Ex. 5.5 in \cite{LDGibsMeasRasoulAga2005}) and for $ \kappa \in (0,1)$ and $\mu_1, \mu_2 \in \mathscr{M}_1(\R)$ with $m_V( \mu_i) < \infty $, we have that $m_V( (1-\kappa) \mu_1+ \kappa \mu_2) = (1- \kappa) m_V( \mu_1) + \kappa m_V( \mu_2)$. 
\end{proof}

\begin{proof}(of Theorem \ref{ThmLDPEmpMeasuresOrlicz}) 
	We are going to prove Theorem \ref{ThmLDPEmpMeasuresOrlicz} by applying the version of the G\"artner--Ellis theorem stated in Lemma \ref{LemGaertnerEllis}. We view $( \mathcal{L}^{n,V} )_{n \in \N}$ as $\mathcal{M}(\R)$-valued random variables, where $\mathcal{M}(\R)$ is the space of finite signed measures on the Borel $\sigma$-algebra on $\R$. Let $\mathcal{M}(\R)' $ denote the algebraic dual of $\mathcal{M}(\R)$, i.e., the space of linear functions mapping from $\mathcal{M}(\R)$ to $\R$. Recall that $C_b( \R)$ denotes the space of continuous and bounded functions mapping from $\R$ to $\R$. We define 
	\begin{equation*}
	\mathcal{Y} := \Big \{ \widehat{ \lambda } \in \mathcal{M}(\R)' \ : \ \exists \lambda \in C_b( \R) \, : \, \forall \mu \in \mathcal{M}(\R) \ : \ \widehat{\lambda} ( \mu ) = \int_{ \R } \lambda( x) \mu(d x) \Big \} \subseteq \mathcal{M}(\R)'.
	\end{equation*}
	The set $\mathcal{Y}$ is separating in $ \mathcal{M}(\R)$, i.e., for $ \nu , \mu \in \mathcal{M}(\R)$ with $ \mu \neq \nu$, we find a $\widehat{ \lambda } \in \mathcal{Y}$ such that $ \widehat{ \lambda }( \mu - \nu ) \neq 0$ (see \cite[p. 261]{DZ2011}).  The $\mathcal{Y}$-topology on $\mathcal{M}(\R)$ is the topology generated by the system of sets 
	\begin{equation*}
	U_{ \delta, x, \widehat{ \lambda} } := \big \{  \mu \in \mathcal{M}(\R) \ : \ \big | \widehat{ \lambda }( \mu ) -x \big | < \delta   \big \},
	\end{equation*}
	where $ \widehat{ \lambda } \in \mathcal{Y}, x \in \R$ and $\delta \in (0, \infty)$. By definition, the relative topology on $\mathcal{M}_1(\R) \subseteq \mathcal{M}( \R)$ is the weak topology. By \cite[Theorem B.8]{DZ2011} we have that $\mathcal{M}( \R)$ together with the $\mathcal{Y}$-topology is a locally convex, Hausdorff topological vector space and $\mathcal{Y}$ is the topological dual $\mathcal{M}( \R)^*$ of $\mathcal{M}( \R)$. Thus, we can identify $\mathcal{M}(\R)^*$ with $C_b(\R)$. Now the framework for Lemma \ref{LemGaertnerEllis} is set up and we can compute the Gärtner--Ellis limit
	\begin{equation}
	\label{EqGarEllLimit}
	\Lambda_V( \lambda) : = 
	\lim_{n \rightarrow \infty } \frac{1}{n} \log \mathbb{E} \Big[ \exp \Big ( n \int_{\R} \lambda(x) \mathcal{L}^{n,V}(dx) \Big ) \Big] , \quad \lambda \in C_b( \R),
	\end{equation}
	and then show that it is Gateaux-differentiable. This can be done by using an appropriate exponential tilting argument. In the following, we denote by $X_R^{n,V}$ a random variable uniformly distributed on $\mathbb{B}^{n,V}_R$. Then, for the quantity in the Gärtner--Ellis limit \eqref{EqGarEllLimit}, where $\lambda \in C_b( \R)$, we have 
	\begin{align}
	\nonumber
	\frac{1}{n} \log \mathbb{E} \Big[ \exp \Big (n  \int_{\R} \lambda(x)  \mathcal{L}^{n,V}(dx) \Big ) \Big] & = \frac{1}{n} \log \mathbb{E} \Big[ \exp \Big ( \sum_{i=1}^{n} \lambda \Big( X_{R}^{n,V}(i) \Big) \Big ) \Big] \\
	\label{EqGEExprOrl}
	& = \frac{1}{n} \log  \int_{\mathbb{B}^{n,V}_R} \exp \Big ( \sum_{i=1}^{n} \lambda ( x_i ) \big) d(x_1,\dots,x_n) - \frac{1}{n} \log \vol_n\big(\mathbb{B}^{n,V}_R \big).
	\end{align} 
	We now consider a sequence of iid random variables $(Z_i)_{i \in \N}$, all distributed according to the Gibbs density $p:\R\to\R$,
	\begin{equation*}
	p(x): = e^{ \alpha( \lambda ) V(x) + \lambda(x) - \varphi_V( \alpha( \lambda ), \lambda )} , 
	\end{equation*}
	where $ \alpha( \lambda)$ and $ \varphi_V( \alpha( \lambda), \lambda)$ are defined as in Lemma \ref{AuxLemmaExpTilt}. Now, we define a new sequence of independent and identically distributed random variables, $Y_i := V(Z_i) - R$, $i \in \N$. We obtain the identity
	\begin{align*}
	& \int_{\mathbb{B}^{n,V}_R} \exp \Big ( \sum_{i=1}^{n} \lambda ( x_i ) \Big) d(x_1,\dots,x_n) \cr 
	&= \int_{\mathbb{B}^{n,V}_R} \exp \Big ( \sum_{i=1}^{n} \lambda ( x_i ) \Big) \prod_{i=1}^{n} p(x_i)^{ -1 }\prod_{i=1}^{n} p(x_i) \, d(x_1,\dots,x_n) \\
	& = \mathbb{E} \Big[ \mathbb{1}_{ \{  ( Z_1,...,Z_n) \in \mathbb{B}^{n,V}_R \} } \exp \Big ( \sum_{i=1}^{n} \lambda ( Z_i ) - \alpha( \lambda ) \sum_{i=1}^{n} V( Z_i) - \sum_{i=1}^{n} \lambda(Z_i) + n \varphi_V( \alpha( \lambda) , \lambda) \Big) \Big] \\
	& = \exp \big( n \varphi_V( \alpha( \lambda ), \lambda) - n \alpha( \lambda) R \big)\cdot \mathbb{E} \Big[ \mathbb{1}_{ \{ \sum_{i=1}^{n} Y_i \leq 0 \} } \exp \Big (  - \alpha( \lambda ) \sum_{i=1}^{n} Y_i \Big) \Big].
	\end{align*}
	The expectation in the previous line is bounded from above, since $\alpha(\lambda ) \in (-\infty, 0)$, more precisely, 
	\begin{equation*}
	\mathbb{E} \Big[ \mathbb{1}_{ \{ \sum_{i=1}^{n} Y_i \leq 0 \} } \exp \Big (  - \alpha( \lambda ) \sum_{i=1}^{n} Y_i \Big) \Big] \leq 1.
	\end{equation*}
	Therefore, we obtain the upper bound
	\begin{equation}
	\label{EqUpperBound}
	\frac{1}{n} \log 
	\int_{\mathbb{B}^{n,V}_R} \exp \Big ( \sum_{i=1}^{n} \lambda ( x_i ) \Big) \, d(x_1,\dots,x_n) \leq  \varphi_V( \alpha( \lambda ), \lambda) -  \alpha( \lambda) R . 
	\end{equation}
	For a lower bound, we fix some $c \in ( 0, \infty)$ and obtain
	\begin{align*}
	\mathbb{E} \Big[ \mathbb{1}_{ \{ \sum_{i=1}^{n} Y_i \leq 0 \} } e^{  - \alpha( \lambda ) \sum_{i=1}^{n} Y_i  } \Big] & \geq \mathbb{E} \Big[ \mathbb{1}_{ \{ - c \sqrt{n}  \leq  \sum_{i=1}^{n} Y_i \leq 0 \} } e^{  - \alpha( \lambda ) \sum_{i=1}^{n} Y_i  } \Big] \\
	& \geq 
	\mathbb{E} \Big[ \mathbb{1}_{ \{ - c \sqrt{n} \leq  \sum_{i=1}^{n} Y_i \leq 0 \} } e^{ \alpha( \lambda ) c \sqrt{n}} \Big]   \\
	& = e^{ \alpha( \lambda ) c \sqrt{n}} \mathbb{P} \Big[ - c \sqrt{n} \leq  \sum_{i=1}^{n} Y_i \leq 0 \Big]  .
	\end{align*}
	This yields the lower bound
	\begin{align}\label{EqlowerBound}
	& \frac{1}{n} \log 
		\int_{\mathbb{B}^{n,V}_R} \exp \Big ( \sum_{i=1}^{n} \lambda ( x_i ) \Big) d(x_1,\dots,x_n) \cr
		& \geq \frac{c \alpha( \lambda ) }{\sqrt{n}} + \frac{1}{n} \log \mathbb{P} \big[ - c \sqrt{n} \leq  \sum_{i=1}^{n} Y_i \leq 0 \big] + \varphi_V( \alpha( \lambda ), \lambda) -  \alpha( \lambda) R. 
	\end{align}
	By Equation \eqref{EqExpValueVofZ} and Equation \eqref{EqVarianceofVofZ} in Lemma \ref{AuxLemmaExpTilt}, we have that $ \E [ Y_1] = \E[ V( Z_1)] - R = 0$ and $ \mathbb{V}[ Y_1] = \mathbb{V}[V(Z_1)] \in (0, \infty)$. Hence, we can apply the central limit theorem to the sequence $(Y_i)_{i \in \N}$. Thus, the probability in Equation \eqref{EqlowerBound} tends to $ \mathcal{N}( 0, \mathbb{V}[ Y_1])( [-c, 0]) \in (0, 1 )$ as $ n \rightarrow \infty$. Combining the upper bound \eqref{EqUpperBound} and the lower bound \eqref{EqlowerBound}, we get
	\begin{equation}\label{eq:limit orlicz ball integral}
	\lim_{ n \rightarrow \infty }\frac{1}{n} \log \int_{\mathbb{B}^{n,V}_R} \exp \Big ( \sum_{i=1}^{n} \lambda ( x_i ) \Big) d(x_1,\dots,x_n)  = \varphi_V( \alpha( \lambda ), \lambda ) - \alpha ( \lambda ) R. 
	\end{equation}
	In case of $\lambda \equiv 0$, we receive the asymptotic volume of an Orlicz ball with radius $R $, i.e.,
	\begin{equation}\label{eq:limit log volume orlicz}
	\lim_{n \rightarrow \infty} \frac{1}{n} \log \vol_n \big( \mathbb{B}^{n,V}_R  \big) = \lim_{ n \rightarrow \infty }\frac{1}{n} \log \int_{\mathbb{B}^{n,V}_R} d(x_1,\dots,x_n)  = \varphi_V( \alpha( 0 ), \lambda ) - \alpha ( 0) R .
	\end{equation}
	Putting \eqref{eq:limit orlicz ball integral} and \eqref{eq:limit log volume orlicz} together, we obtain from \eqref{EqGEExprOrl} that the Gärtner--Ellis limit $\Lambda_V: C_b( \R ) \rightarrow \R $ (see \eqref{EqGarEllLimit}) is given by
	\begin{equation}
	\label{EqGELimCumGenFct}
	\Lambda_V( \lambda) = 
	\varphi_V( \alpha( \lambda ), \lambda ) - \alpha ( \lambda ) R -\varphi_V( \alpha( 0 ), \lambda ) + \alpha ( 0) R.
	\end{equation}
	In particular, by Lemma \ref{AuxLemmaExpTilt}, $\Lambda_V$ is Fréchet-differentiable. In order to be able to apply Lemma \ref{LemGaertnerEllis}, it remains to show that the sequence of empirical measures $( \mathcal{L}^{n,V} )_{n \in \N}$ is exponentially tight. We define the quantity $\widetilde{M} := \sup_{ |x| \geq 1} \frac{|x|}{V(x)}$ and note that $\widetilde{M} \in (0, \infty )$, since $V$ is an Orlicz function. We set $M := 1+ \widetilde{M} R$ and let $X^{n , V}_R \sim \U ( \mathbb{B}^{n,V}_R)$, for $n \in \N$. Then, recalling that $m_1$ denotes the first moment map, we have 
	\begin{align*}
	m_1(\mathcal{L}^{n,V}) & = \frac{1}{n} \sum_{i =1 }^{n} |X_{R}^{n,V}(i)| \\
	& = \frac{1}{n} \sum_{i =1 }^{n} |X_{R}^{n,V}(i)| \mathbb{1}_{ \{ |X_{R}^{n,V}(i) | \leq 1 \} } +  \frac{1}{n} \sum_{i =1 }^{n} |X_{R}^{n,V}(i)| \mathbb{1}_{ \{ |X_{R}^{n,V}(i) | >  1 \} } \\
	& \leq 1 + \frac{1}{n} \sum_{i =1 }^{n} V\big ( X_{R}^{n, V}(i) \big )  \frac{|X_{R}^{n, V}(i)|}{V\big ( X_{R}^{n, V}(i) \big )} \mathbb{1}_{ \{ |X_{R}^{n, V}(i) | >  1 \} } \\
	& \leq 1 + \widetilde{M} \frac{1}{n} \sum_{i =1 }^{n} V\big (X_{R}^{n, V}(i) \big ) \\
	& \leq 1 + \widetilde{M} R = M .
	\end{align*}
	Thus, for every $n \in \N$, we have $ \mathcal{L}^{n,V}\in K_1^M $, where $K_1^M$ is the weakly compact set from Proposition \ref{PropCompactSetpWasserst}. In other words, we have that
	\begin{equation*}
	\frac{1}{n} \log \mathbb{P} \big[ \mathcal{L}^{n,V} \in \big(K_1^M \big)^c  \big] = - \infty ,
	\end{equation*} 
	implying the exponential tightness of $( \mathcal{L}^{n,V})_{n \in \N}$.
	It now follows from Lemma \ref{LemGaertnerEllis} that $( \mathcal{L}^{n,V})_{n \in \N}$ satisfies an LDP in the weak topology on $ \mathscr{M}_1(\R)$. The GRF is the Legendre transform of $\Lambda_V$, which is, by Proposition \ref{PropLegTransformOrlicz}, the strictly convex function $\mathbb{I}_V : \mathscr{M}_1( \R ) \rightarrow [0, \infty ]$ with
	\begin{equation*}
	\mathbb{I}_V(\mu) = \begin{cases} 
	H( \mu | \mu_{ V, \alpha(0) }) + \alpha(0) \big[ m_V( \mu ) -R  \big] &: m_V( \mu ) \leq R \\
	\infty & \text{else}.  
	\end{cases}
	\end{equation*} 
	Now consider the situation when there exists another Orlicz function $W$ such that
	\begin{equation*}
	\frac{V(x)}{W(x)} \longrightarrow \infty \quad \text{as} \quad |x| \rightarrow \infty. 
	\end{equation*}
	Recall the set $ K_{V}^R$ from Proposition \ref{PropCompactSetGenWasserst} and note that $ ( \mathcal{L}^{n,V} )_{n \in \N} \subseteq K_{V}^R$. Since $K_{V}^R $ is compact in $ \mathscr{M}^W_1( \R )$ by Proposition \ref{PropCompactSetGenWasserst}, $( \mathcal{L}^{n,V})_{n \in \N}$ is exponentially tight in $\mathscr{M}^W_1( \R )$. By Corollary \ref{CorCoarserFinerTop}, $(\mathcal{L}^{n,V})_{n \in \N}$ satisfies the LDP in $ (\mathscr{M}^W_1( \R ), d_W )$ with strictly convex GRF $\mathbb{I}_V$. This completes the proof.
\end{proof}

\begin{cor}
	\label{CorThinshellOrl}
	Let $V$ be an Orlicz function and $p \in (1,\infty)$ such that 
	\begin{equation*}
	\frac{V(x)}{|x|^p} \longrightarrow \infty \quad \text{as} \quad |x| \rightarrow \infty 
	\end{equation*}
	and let $ X^{n ,V}_R \sim \U ( \mathbb{B}_R^{n,V })$.
	Then, for $ 1 \leq q < p $, we have that $ ( n^{-1/p} || X^{n,V}_R ||_p )_{n \in N}$ satisfies an LDP in $\R$ with GRF $ J_p^{V} : \mathbb{R} \rightarrow [0, \infty ]$, where
	\begin{equation*}
	J_p^{V} ( x) := 
	\begin{cases} 
	-\max_{ \alpha < 0 }  \big \{ \alpha - \log \int_{\R } e^{ \alpha V(t)} dt   \big \} - \max \big \{ h( \nu ) : \nu \in \mathscr{M}_1(\R) , m_2( \nu )^{1 /p} = x , m_V( \nu ) \leq  R  \big \} &: x \geq 0  \\
	\infty  &: x <0.
	\end{cases}
	\end{equation*}
\end{cor}

\begin{rem}
	In case of $ p=2 $ and $R=1$, the GRF $J_2^V$ from Corollary \ref{CorThinshellOrl} is the same as the GRF from \cite[Section $3$]{kimliaoramanan2020asymptotic}. There, the authors used a different approach, but encountered the same growth condition on $V$, i.e., they assumed
	\begin{equation*}
	\frac{V(x)}{|x|^2} \longrightarrow \infty \quad \text{as} \quad |x| \rightarrow \infty .
	\end{equation*}
	This is necessary to introduce the exponential change of measure used in the proof of Proposition 3.11 in \cite{kimliaoramanan2020asymptotic}. In \cite{kimliaoramanan2020asymptotic} this condition may seem artificial at first, but looking at the proof of Theorem \ref{ThmLDPEmpMeasuresOrlicz}, one sees that this condition appears in a quite natural way.     
\end{rem}

\begin{proof}(of Corollary \ref{CorThinshellOrl})
	Under our assumption, by Theorem \ref{ThmLDPEmpMeasuresOrlicz}, $( \mathcal{L}^{n,V} )_{n \in \N}$ satisfies an LDP in $ \mathscr{M}^p_1( \R)$ with strictly convex GRF $\mathbb{I}_V : \mathscr{M}_1( \R) \rightarrow [0, \infty]$, where
	\begin{equation*}
	\mathbb{I}_V(\mu) = \begin{cases} 
	H( \mu | \mu_{ V, \alpha(0) }) + \alpha(0) \big[ m_V( \mu ) -R  \big] &: m_V( \mu ) \leq R \\
	\infty & \text{else}.  
	\end{cases}
	\end{equation*}
	We can apply the contraction principle to the continuous mapping $ m_p: \mathscr{M}^p_1( \R ) \rightarrow [0,\infty) $ with $m_p( \nu) = \int_{ \R } |x|^p d \nu( x) $ (see Remark \ref{RemContinuityVthmomentmap}) in order to establish an LDP for $ ( n^{-1} || X^{ n,V}_R ||_p^p )_{n \in \N }$ with GRF $\widetilde{J}_p^{V} : \R \rightarrow [0,\infty) $, given as
	\begin{equation*}
	\widetilde{J}_p^{V} (x ) : = 
	\begin{cases}
	\inf \big \{  H( \mu | \mu_{ V, \alpha(0 ) }) + \alpha( 0) [ R - m_V( \mu )] :\mu \in \mathscr{M}_1( R), m_V( \mu ) \leq R, m_p( \mu ) = x   \big \} &:  x \geq 0 \\
	\infty & \text{else}.
	\end{cases} 
	\end{equation*}
	Further, because the relative entropy appears, in the infimum we only need to consider measures $\mu$ which are absolutely continuous with respect to $\mu_{ V, \alpha(0)}$ (and hence absolutely continuous with respect to Lebesgue measure on $\R$). By Remark \ref{RemGRFOrl}, we get
	\begin{align*}
	\widetilde{J}_p^{V} (x ) & = \inf \big \{  -h( \mu ):\mu \in \mathscr{M}_1(\R ), m_V( \mu ) \leq R, m_p( \mu ) = x   \big \}  + \varphi_V( \alpha( 0), 0) - \alpha( 0) R \\
	& = - \sup \big \{  h( \mu ):\mu \in \mathscr{M}_1(\R ), m_V( \mu ) \leq R, m_p( \mu ) = x   \big \}  + \varphi_V( \alpha( 0), 0) - \alpha( 0) R,
	\end{align*}
	where we note that the supremum is a maximum, since it is of the form given in Equation \eqref{EqRestrMaxEntr} in the maximum entropy principle (see Lemma \ref{LemMaxEntr}). Another application of the contraction principle to $ ( n^{-1} || X^{ n,V}_R ||_p^p )_{n \in \N }$ via the continuous mapping $ x \mapsto x^{1/p}$ gives us the LDP for $ ( n^{-1/p} || X^{ n,V}_R ||_p )_{n \in \N }$ with the GRF $ J_p^{V} : \R \rightarrow [0,\infty)$ given by
	\begin{equation*}
	J_q^{V} ( x) = 
	\begin{cases} 
	-\max_{ \alpha < 0 }  \big \{ \alpha - \log \int_{\R } e^{ \alpha V(t)} dt   \big \} - \max \big \{ h( \nu ) : \nu \in \mathscr{M}_1(\R ) , m_p( \nu )^{1 /p} = x , m_V( \nu ) \leq  R  \big \} &: x \geq 0  \\
	\infty  &: x <0.
	\end{cases} 
	\end{equation*}
	Here we used Lemma \ref{AuxLemmaExpTilt}, where we showed that 
	\begin{equation*}
	\varphi_V( \alpha( 0), 0) - \alpha( 0) R = - \max_{ \alpha < 0 } \Big \{ \alpha R - \log \int_{\R } e^{ \alpha V(t)} dt   \Big \} .
	\end{equation*} 
	This completes the proof.  
\end{proof}

\section{Proof of Theorem \ref{ThmLimCondDist}} 

We prove Theorem \ref{ThmLimCondDist} by employing the Gibbs conditioning principle (see Proposition \ref{PropGibbsCond}). This leads to the maximization of the entropy function $h$ under certain moment constraints.
\par{}
For an Orlicz function $V$ and $\alpha \in (-\infty , 0)$, we recall the Gibbs distribution $ \mu_{V, \alpha }$ with potential $V$ and critical temperature $\alpha$, which is given by the Lebesgue density
\begin{equation*}
\frac{d \mu_{ V, \alpha }}{dx}(x) =  e^{ \alpha V(x)  - \varphi_V( \alpha, 0) } , \quad x \in \R.
\end{equation*}
Here, $e^{ - \varphi_V( \alpha, 0) }$ is the normalizing constant, i.e., $ \varphi_V( \alpha, 0 ) = \log \int_{\R} e^{ \alpha V(t)} dt $. 
Note that in Section 4, $\alpha = \alpha( \lambda)$ depends implicitly on a parameter $ \lambda$. In the following, we will only face the situation where $ \lambda \equiv 0$ and $ \alpha = \alpha( R)$ depends on a parameter $R \in (0, \infty)$.

\begin{lem}
	\label{LemMonotonepMoment}
	Let $V_1, V_2$ be two Orlicz functions that satisfy both
	\begin{equation*}
	\frac{V_1(x)}{V_2(x)} \longrightarrow \infty  \quad \text{ as } \quad  |x| \longrightarrow \infty \qquad \text{and} \qquad 	\int_{\R } V_1(x) e^{ \alpha V_2(x)} dx < \infty, 
	\end{equation*}
	for all $ \alpha \in  (-\infty, 0 )$.
	Then, there exists an $ \overline{R}  \in (0,\infty)$ such that for all $0 < R \leq \overline{R}$, we have
	\begin{equation}
	\label{EqIneqpthMoment}
	m_{V_1}( \mu_{ V_2, \alpha( R )}) \leq m_{V_1}( \mu_{ V_2, \alpha(\overline{R} )} )= 1,
	\end{equation}
	where $ \alpha( R )$ is the solution of
	\begin{equation}
	\label{EqImplicitbeta}
	R = \int_{ \R } V_2(x)  \mu_{V_2, \alpha( R )} (dx).
	\end{equation} 
\end{lem}

\begin{rem}
The quantity $ \alpha(R ) \in (-\infty, 0)$ from Equation \eqref{EqImplicitbeta} exists and is unique by Lemma \ref{AuxLemmaExpTilt}.
\end{rem}

\begin{proof}(of Lemma \ref{LemMonotonepMoment})
    Analogously to the first part of the proof of Lemma \ref{AuxLemmaExpTilt}, we can show that there exists an $ \overline{ \alpha} \in (-\infty, 0)$ such that
    \begin{equation*}
    \int_{ \R  } V_1(x) e^{ \overline{ \alpha}  V_2(x) - \varphi_{V_2}( \overline{\alpha} , 0) } dx = 1.
    \end{equation*}   
    We define 
    \begin{equation}
    \label{EqDefRbar}
    \overline{ R }  := \int_{ \R  } V_2(x) e^{ \overline{ \alpha}  V_2(x) - \varphi_{V_2}( \overline{\alpha} , 0) } dx  \in (0, \infty ).
    \end{equation}
    Since the solution $\overline{\alpha}$ of Equation \eqref{EqDefRbar} is unique by Lemma \ref{AuxLemmaExpTilt}, we can write $ \overline{\alpha} =: \alpha( \overline{R})$. Now take some $ R > 0$ with $ R \leq \overline{ R } $. For the solution $ \alpha( R ) $ of Equation \eqref{EqImplicitbeta}, we have that $ \alpha( R) \leq \alpha (\overline{R }) $. This holds since the map $ \alpha \mapsto \int_{ \R  } V_2(x) e^{  \alpha  V_2(x) - \varphi_{V_2}( \alpha , 0) } dx $ is monotone increasing (see, e.g., p. 5 in \cite{PJ2020OrliczLim}). Thus, if we are able to show that the function
	\begin{equation*}
	\alpha \mapsto \int_{ \R  } V_1(x) e^{  \alpha  V_2(x) - \varphi_{V_2}( \alpha , 0) } dx
	\end{equation*}
	is monotone increasing on $(-\infty,0)$, we obtain the claim in Equation \eqref{EqIneqpthMoment} since $\alpha(R) \leq \alpha( \overline{ R })= \overline{ \alpha}$. Take $ \alpha_1 , \alpha_2 $ with $- \infty < \alpha_1 \leq \alpha_2 < 0$. Then, there exists an $ \overline{x} \geq 0$ such that
	\begin{equation}
	\label{EqPropertyXbar1}
	e^{ \alpha_1 V_2(x) - \varphi_{V_2}( \alpha_1 , 0 ) } \geq e^{ \alpha_2 V_2(x) - \varphi_{V_2}( \alpha_2 , 0 ) }  \quad \text{for} \quad x \in [0, \overline{x} ]
	\end{equation}
     and
	\begin{equation}
	\label{EqPropertyXbar2}
	e^{ \alpha_1 V_2(x) - \varphi_{V_2}( \alpha_1 , 0 ) } \leq e^{ \alpha_2 V_2(x) - \varphi_{V_2}( \alpha_2 , 0 ) } \quad \text{for} \quad x >  \overline{x}.
	\end{equation}
	In the following, we denote by $V_1^{-1}$ the inverse of $V_1$ restricted on the positive real numbers, which exists since $V_1$ is an Orlicz function. Then we get
	\begin{align*}
	& \int_{ \R  } V_1(x) e^{ \alpha_1 V_2(x) - \varphi_{V_2}( \alpha_1 , 0 ) } dx\cr 
	& = 2 \int_{ 0 }^{ \infty } \int_{ V_1(x)}^{ \infty } e^{ \alpha_1 V_2(t) - \varphi_{V_2}( \alpha_1 , 0 ) }  dt dx  \\
	& = 2 \int_{ 0 }^{ V_1^{-1}( \overline{x})} \int_{ V_1(x)}^{ \infty } e^{ \alpha_1 V_2(t) - \varphi_{V_2}( \alpha_1 , 0 ) }  dt dx + 2 \int_{ V_1^{-1}(\overline{x})}^{ \infty } \int_{ V_1(x)}^{ \infty } e^{ \alpha_1 V_2(t) - \varphi_{V_2}( \alpha_1 , 0 ) } dt dx \\
	&  = 2 \int_{ 0 }^{ V_1^{-1}( \overline{x})} \Big[  \frac{1}{2} - \int_{ 0}^{ V_1(x)} e^{ \alpha_1 V_2(t) - \varphi_{V_2}( \alpha_1 , 0 ) } dt \Big ]  dx + 2 \int_{ V_1^{-1} (\overline{x})}^{ \infty } \int_{ V_1(x)}^{ \infty } e^{ \alpha_1 V_2(t) - \varphi_{V_2}( \alpha_1 , 0 ) } dt dx \\
	& \stackrel{ \eqref{EqPropertyXbar1}, \eqref{EqPropertyXbar2}}{\leq}
	2 \int_{ 0 }^{ V_1^{-1}(\overline{x})} \Big[  \frac{1}{2} - \int_{ 0}^{ V_1(x) } e^{ \alpha_2 V_2(t) - \varphi_{V_2}( \alpha_2 , 0 ) } dt \Big ]  dx + 2 \int_{ V_1^{-1}(\overline{x})}^{ \infty } \int_{ V_1(x)}^{ \infty } e^{ \alpha_2 V_2(t) - \varphi_{V_2}( \alpha_2 , 0 ) } dt dx \\
	& = \int_{ \R  } V_1(x) e^{ \alpha_2 V_2(x) - \varphi_{V_2}( \alpha_2 , 0 ) } dx.
	\end{align*}
\end{proof}

\begin{rem}
	\label{RemMonotoneVthMomentmap}
	In the proof of Lemma \ref{LemMonotonepMoment} we have shown that $ \alpha \mapsto m_{V_1}( \mu_{ V_2, \alpha })$ is monotone increasing. The same argument can be used to prove that $ \alpha \mapsto m_{V_2}( \mu_{ V_1, \alpha })$ and $ \alpha \mapsto m_{V_1}( \mu_{ V_1, \alpha })$ are monotone increasing. Here, note that the latter two functions are well-defined even if we do not assume that
	\begin{equation*}
	\int_{\R} V_1(x) e^{ \alpha V_2(x) - \varphi_{V_2}( \alpha, 0)} dx < \infty 
	\end{equation*}
	for all $\alpha \in (-\infty, 0 )$.
\end{rem}

In order to get a nice geometric interpretation of Theorem \ref{ThmLimCondDist}, we need the following limit theorem from \cite{PJ2020OrliczLim}. We recall that $ \lambda_{R,V}^{(n,k)}$ is the distribution of the first $k$ components of a random vector uniformly distributed on the Orlicz ball $ \mathbb{B}_R^{n,V}$, where $R \in (0, \infty )$. The distribution $\lambda_{R,V}^{(n,k)} $ has a Lebesgue density on $\R^k$ (see p.6 in \cite{PJ2020OrliczLim}) which we shall denote by $\frac{d \lambda_{R,V}^{(n,k)}}{dx}$. Further, for an Orlicz function $V$ and $\alpha \in (-\infty,0)$, we denote by $  \mu_{V, \alpha }^{ \otimes k}$ the $k$-fold product measure of $ \mu_{V, \alpha }$.

\begin{proposition}(Theorem A in \cite{PJ2020OrliczLim})
	\label{PropLimUnifOrlicz}
	Let $V$ be an Orlicz function, $\alpha \in (- \infty , 0)$, $R \in (0, \infty )$, $k \in \N$.
    Then, there exists a constant $C \in (0, \infty )$ such that
	\begin{equation}
	\label{EqTVmugamma}
	\int_{ \R^k} \Big | \frac{d \lambda_{R,V}^{(n,k)}}{dx} (x) -  \frac{d\mu_{V, \alpha(R)}^{ \otimes k } (x)} { dx}  \Big | d x \leq \frac{C k}{n},
	\end{equation}
	where $ \alpha( R) \in (- \infty, 0)$ is the unique solution of
	\begin{equation}
	\label{EqImplicitalpha}
	R = \int_{ \R } V(x) e^{ \alpha(R) V(x) - \varphi_V( \alpha(R) , 0)} d x.
	\end{equation} 
\end{proposition} 

\begin{rem}
	\label{RemConDistOrlicGibbs}
	In particular, Equation \eqref{EqTVmugamma} in Proposition \ref{PropLimUnifOrlicz} implies convergence in distribution of $  \DistkKompOrl  $ towards $ \mu_{V, \alpha(R)} ^{ \otimes k }$ as $n \rightarrow \infty$.
\end{rem}

\begin{rem}
	In \cite{PJ2020OrliczLim}, the authors considered a more general class of Orlicz functions, for our purpose the version stated Proposition \ref{PropLimUnifOrlicz} suffices.
\end{rem}

\begin{proposition} 
\label{PropMainRes}
Let $V_1$ and $V_2$ be two Orlicz functions which satisfy
	\begin{equation*}
	\frac{V_1(x)}{V_2(x)} \longrightarrow \infty  \quad{} \text{as} \quad{} |x| \rightarrow \infty. 
	\end{equation*}
Suppose that $X_1^{ n,V_1 } \sim \U(\mathbb{B}^{n,V_1}_{1})$ and denote by $( \mathcal{L}^{n,V_1})_{n \in \N}$ the associated sequence of empirical measures. For $R \in (0, \infty )$, we define $C := [0, R ] \subseteq \R$ and for $ \eps > 0$, let $C_{ \eps } := [ - \eps , R + \eps ]$. Then, the measure
	\begin{equation}
	\label{EqOptMeas}
	\nu_{*} := \arg \max \big \{ h( \nu ) : m_{V_1}( \nu ) \leq 1, m_{V_2}( \nu ) \in C \big \} ,
	\end{equation}  
	which is the measure maximizing the entropy under the constraints $m_{V_1}( \nu ) \leq 1$, $ m_{V_2}(\nu) \in C$, exists, is unique and
	\begin{equation}
	\label{EqPropLimitEmpMeasNu}
	\lim_{ \eps \rightarrow 0} \lim_{n \rightarrow \infty } \prb[ \mathcal{L}^{n,V_1} \in \cdot | m_{V_2}( \mathcal{L}^{n,V_1} ) \in C_{ \eps}   ] = \delta_{ \nu_{ *}}( \cdot ),
	\end{equation}
	where $ \delta_{ \nu_{*} } : \mathcal{M}_1( \R) \rightarrow \R $ with
	\begin{equation*}
	\delta_{ \nu_{*} }( \mu ) = \begin{cases}
	1 & : \mu = \nu_{*} \\
	0 & \text{else}.
	\end{cases}
	\end{equation*} 
	Moreover, for $k \in \N $,
	\begin{equation}
	\label{EqPropMeasNu} 
	\prb \Big [ \big  ( X^{ n,V_1}_1(1),..., X^{ n,V_1 }_1(k) \big ) \in \cdot \Big | \frac{1}{n} \sum_{ i=1}^n V_2 \big ( X^{n,V_1}_1(i) \big )  \in C_{ \eps } \Big ]  \stackrel{d}{ \longrightarrow } \nu_{ *}^{ \otimes k} ,
	\end{equation}
	as $n \rightarrow \infty $ followed by $\eps \rightarrow 0$ and $\nu_{ *}^{ \otimes k}$ is the $k$-fold product measure of $\nu_{ *}$. 
\end{proposition}

\begin{proof}
	For the proof we mainly use the Gibbs conditioning principle from Proposition \ref{PropGibbsCond}. Due to our assumptions, $ (\mathcal{L}^{n,V_1})_{ n \in \N }$ satisfies an LDP in $(\mathscr{M}^{V_2}_1( \R ), d_{V_2} ) $ by Theorem \ref{ThmLDPEmpMeasuresOrlicz}. Further, we recall that the map $m_{V_2}: \mathscr{M}^{V_2}_1( \R ) \rightarrow \R $ is continuous. Then, for $\eps > 0$, the set $ F_{ \eps } :=\big  \{ \nu \in \mathscr{M}^{V_2}_1( \R ) \ | \ m_{V_2}( \nu ) \in C_{ \eps } \big \} = m_{V_2}^{-1}( [  - \eps , R + \eps ]) $ is closed in the $d_{V_2}$-topology on $\mathscr{M}^{V_2}_1( \R )$. Furthermore, the set $F:= \bigcap_{ \eps > 0} F_{ \eps }$ is closed as an intersection of closed sets and, for $\eps > 0$, we have 
	\begin{equation*}
	F = m_{V_2}^{-1}(C )  \subseteq m_{V_2}^{-1} ( C_{ \eps }^{ \circ }) \subseteq [  m_{V_2}^{-1} ( C_{ \eps } )]^{ \circ } = F_{ \eps }^{ \circ },
	\end{equation*}
	where the second inclusion holds since $m_{V_2}$ is continuous. This shows properties $(2)$ and $(4)$ of Proposition \ref{PropGibbsCond}. For Property $(3)$ we need to check that $ \prb [  \mathcal{L}^{n,V_1} \in F_{ \eps } ]> 0$ for all sufficiently large $n \in \N$ and all $\eps > 0$. Since $ (\mathcal{L}^{n,V_1})_{ n \in \N }$ satisfies an LDP in $(\mathscr{M}_1^{V_2} (\R),  d_{V_2} )$ with strictly convex GRF $ \mathbb{I}_{V_1}$ given in Equation \eqref{EqGRFLDPempMOrl}, we have that
	\begin{equation*}
	\liminf_{n \rightarrow \infty } \frac{1}{n} \log \prb [ \mathcal{L}^{n,V_1}  \in F_{ \eps } ] \geq - \inf_{ \nu \in F^{ \circ }_{ \eps }} \mathbb{I}_{V_1}( \nu ) \geq - \inf_{ \nu \in F} \mathbb{I}_{V_1}( \nu ).
	\end{equation*}
	If we are able to show that there exists a measure $\nu \in F $ with $  \mathbb{I}_{V_1}( \nu )  < \infty$, then it follows that $ \prb [  \mathcal{L}^{n,V_1} \in F_{\eps } ]> 0$ for sufficiently large $n \in \N$ and all $\eps > 0$. For $\alpha \in (-\infty,0)$, by Remark \ref{RemMonotoneVthMomentmap}, both mappings $ \alpha \mapsto m_{ V_1} ( \mu_{V_1, \alpha}) $ and $ \alpha \mapsto m_{V_2}( \mu_{ V_1, \alpha})$ are monotone increasing and we have that $ \lim_{ \alpha \rightarrow -\infty} m_{ V_1} ( \mu_{V_1, \alpha}) = \lim_{ \alpha \rightarrow -\infty} m_{ V_2} ( \mu_{V_1, \alpha}) = 0$. Thus, there exists an $ \overline{\alpha} \in (-\infty, 0)$ such that $ m_{V_1}( \mu_{ V_1, \overline{\alpha}}) \leq 1$ and $m_{V_2}( \mu_{ V_1, \overline{\alpha}}) \leq R $ and hence $ \mu_{ V_1, \overline{\alpha}} \in F$. The distribution $ \mu_{ V_1, \overline{ \alpha}}$ is absolutely continuous with respect to Lebesgue measure and thus by Remark \ref{RemGRFOrl} we have   
	\begin{equation*}
	\mathbb{I}_{V_1}( \mu_{ V_1, \overline{\alpha} }) = -h( \mu_{ V_1, \overline{\alpha} }) + c,
	\end{equation*}
	where $c \in \R$ denotes some finite constant and $h$ is the entropy function. For $ -h( \mu_{ V_1, \overline{\alpha} } )$, we have that
	\begin{align*}
	-h( \mu_{ V_1, \overline{\alpha} }) & = \int_{ \R }  e^{ \overline{ \alpha} V_1(x) - \varphi_{V_1}(\overline{ \alpha } , 0)} \log \Big( e^{ \overline{ \alpha} V_1(x) - \varphi_{V_1}(\overline{ \alpha } , 0)}  \Big) dx \\
	& = - \varphi_{V_1}(\overline{ \alpha } , 0) + \overline{ \alpha }  m_{V_1}( \mu_{ V_1, \overline{\alpha} } ) \\
	& <  \infty ,
	\end{align*}
	showing that $ \mathbb{I}_{V_1}(\mu_{ V_1, \overline{\alpha} }) < \infty$.
	\par{}
	For property $(1)$, we observe that $F$ is closed and convex. Indeed, for any $\nu_1, \nu_2 \in F$ and $\lambda \in [0,1]$, we have
	\begin{equation*}
	m_{V_2}( \lambda \nu_1 + (1- \lambda ) \nu_2) = \lambda m_{V_2}( \nu_1) +(1- \lambda ) m_{V_2}( \nu_2 ) \in [ 0 , R ],
	\end{equation*}
	where we used that $ m_{V_2}( \nu_i ) \in [ 0, R]$ for $i=1,2$. Moreover, we have shown that $ \inf_{ \nu \in F} \mathbb{I}_{V_1}( \nu ) < \infty$ and since $\mathbb{I}_{V_1}$ is a strictly convex GRF, $\mathbb{I}_{V_1}$ uniquely attains its minimum over such a set $F$. In total we get property $(1)$ from Proposition \ref{PropGibbsCond} and that $\mathcal{M}_F$ is a singleton with $\mathcal{M}_F = \{ \nu_{*}\}$, where
	\begin{equation*}
	\nu_{ *} = \arg \min \big  \{ \mathbb{I}_{V_1} ( \nu ) \ : \ \nu \in F \big \} = \arg \max \big \{  h ( \nu )  \ : \ \nu \in \mathscr{M}^{V_2}_1( \R ) , m_{V_1}( \nu ) \leq 1 , m_{V_2}( \nu ) \in C \big \} .
	\end{equation*}
	The equality uses Remark \ref{RemGRFOrl} and the observation that in the minimization only those $\nu \in \mathscr{M}^{V_2}_1( \R )$ are relevant which are absolutely continuous with respect to Lebesgue measure and $m_{V_1}( \nu ) \leq 1 $ and $m_{V_2}( \nu ) \in C$. The Gibbs conditioning principle (more precisely \eqref{EqGibbsCondPrSingl}) now gives us the limit in Equation \eqref{EqPropLimitEmpMeasNu}.
	We now want to employ Proposition \ref{propPropchaos} to derive the limit in Equation \eqref{EqPropMeasNu} from the limit 
	\begin{equation*}
	\lim_{ \eps \rightarrow 0} \lim_{n \rightarrow \infty } \prb \big[ \mathcal{L}^{n,V_1} \in \cdot \ \big | \ m_{V_2}( \mathcal{L}^{n,V_1} ) \in C_{ \eps}   \big] = \delta_{ \nu_{ *}}( \cdot ).
	\end{equation*}
	We introduce the family of measures $( \mu^n_{\eps })_{n \in \N , \eps > 0}$, where, for $n \in \N$ and $\eps > 0$, we define
	\begin{equation*}
	\mu_{ \eps }^n( \cdot ) := \prb \big[ X_1 ^{n,V_1} \in \cdot  \ \big |   \ X_1^{n,V_1} \in \mathbb{B}^{n, V_2}_{ R + \eps }\big ].
	\end{equation*}
	Here, for $n \in \N$, $X_1^{n,V_1} \sim \U( \mathbb{B}^{n, V_1}_1)$ (and hence, $ \mathcal{L}^{n,V_1}$ is the empirical measure associated to $X_1^{n,V_1} $). We note that the family $( \mu^n_{ \eps })_{n \in \N , \eps > 0}$ is a family of symmetric probability measures, since, by the symmetry of the Orlicz function $V_1$, $X^{n, V_1}_1$ has a symmetric distribution.
	
	For $n \in \N$ and $\eps > 0$, we consider random variables $X^{n}_{ \eps }$ with distribution $\mu^{n}_{ \eps }$ and let $ \mathcal{L}^{n}_{ \eps}$ be the associated empirical distribution, i.e.,
	\begin{equation*}
	\mathcal{L}^n_{ \eps } : = \frac{1}{n} \sum_{ i=1}^n \delta_{ X_{ \eps }^{n}(i)}.
	\end{equation*} 
	Then, we have
	\begin{align*}
	\delta_{ \nu_{ *}}( \cdot ) & =
	\lim_{ \eps \rightarrow 0} \lim_{n \rightarrow \infty } \prb \big[ \mathcal{L}^{n,V_1} \in \cdot \ \big |   \ m_{V_2}( \mathcal{L}^{n,V_1} ) \in C_{ \eps}   \big] \\
	&= \lim_{ \eps \rightarrow 0} \lim_{n \rightarrow \infty } \prb \big[ \mathcal{L}^{n,V_1} \in \cdot \ \big |   \ X^{n,V_1}_1 \in \mathbb{B}^{n, V_2}_{ R + \eps }   \big] 
	= \lim_{ \eps \rightarrow 0} \lim_{n \rightarrow \infty } \prb \big[  \mathcal{L}^n_{ \eps }  \in \cdot \big].
	\end{align*}
	Here, we used that the event $  m_{V_2}( \mathcal{L}^{n,V_1} ) \in C_{ \eps}  $ can be expressed as $ \ X^{n,V_1}_1 \in \mathbb{B}^{n, V_2}_{ R + \eps }  $. Thus, by Proposition \ref{propPropchaos} and Remark \ref{RemWeakconvergence}, we get that the marginal distribution of the first $k$ coordinates of $ \mu_{\eps}^n$ converges weakly to $ \nu_{ *}^{\otimes k}$. In other words, we have
	\begin{align*}
	\nu_{ *}^{\otimes k}( \cdot ) & =
	\lim_{ \epsilon \rightarrow 0} \lim_{ n \rightarrow \infty } \prb \big  [ ( X^n_{ \eps } (1),..., X^{n}_{ \eps }(k)) \in \cdot \big ] \\
	& = \lim_{ \epsilon \rightarrow 0} \lim_{ n \rightarrow \infty } \prb \big [ ( X^{n,V_1}_{ 1 } (1),..., X^{n,V_1}_{ 1 }(k)) \in \cdot \ \big |   \ X^{n,V_1}_1 \in \mathbb{B}^{n, V_2}_{ R + \eps } \big ] \\
	& = \lim_{ \epsilon \rightarrow 0} \lim_{ n \rightarrow \infty }  \prb \Big [ \big  ( X^{ n,V_1}_1(1),..., X^{ n,V_1 }_1(k) \big ) \in \cdot \Big | \frac{1}{n} \sum_{ i=1}^n V_2 \big ( X^{n,V_1}_1(i) \big )  \in C_{ \eps } \Big ].
	\end{align*}
	This completes the proof.
\end{proof}

In the preceding Proposition \ref{PropMainRes} we have seen that the limiting distribution $\nu_{*}$ follows from a maximum entropy principle. In order to compute $\nu_{*}$, we need the following classical result, which relies on Ex. 5.3 in \cite{ConvexOpt2004} and chapter 12 in \cite{EleOfInf2006}. We state it in the version given in \cite{RKConLim}.

\begin{lem}
	\label{LemMaxEntr}
	Fix $r_i : \R \rightarrow \R$, $\alpha_i \in \R $, for $i=1,...,m$ and $s_j: \R \rightarrow \R $, $\beta_j \in \R $, for $j=1,...,n$. Consider the following maximization problem
	\begin{align}
	\notag
	\max_{ \nu \in \mathscr{M}_1( \R )}  & h( \nu ) \\
	\label{EqRestrMaxEntr}
	 \text{subject to }  & \int_{\R} r_i( x) d \nu( x) = \alpha_i \text{ for } i=1,...,m \\
	\notag
	&  \int_{\R} s_j( x) d\nu(  x) \leq  \beta_j \text{ for } j=1,...,n.
	\end{align}
	Then, a probability measure $\nu_{*} \in \mathscr{M}_1( \R )$ attains the maximum in \eqref{EqRestrMaxEntr} if and only if it is of the following form
	\begin{equation}
		\label{EqEntrResprMaxMeas}
		 \nu_{*}(dx) = \exp \Big( -1 - \kappa^{*}_0 - \sum_{ i=1}^m \lambda_i^{*} r_i(x) - \sum_{ i=1}^n \mu_j^{*} s_j(x) \Big) dx,
	\end{equation}
	with non-negative constants $\kappa_0^{*} $, $( \lambda_i^{*})_{ i=1}^m$, $ ( \mu_j^{*})_{ j=1}^n$ chosen such that $\nu_{*}$ lies in $\mathscr{M}_1( \R )$ and satisfies the constraints in \eqref{EqRestrMaxEntr}. Moreover, if $\nu_{*}$ attains the maximum in \eqref{EqRestrMaxEntr}, then
	\begin{equation}
	\label{EqSlacknessCond}
	\Big(   \int_{\R} s_j( x) \nu_{*}( dx ) - \beta_{j} \Big) \mu_j^{*} =0 , 
	\end{equation}
	for all $j=1,...,n$.
\end{lem}

\begin{proof}(of Theorem \ref{ThmLimCondDist})
In Proposition \ref{PropMainRes} we have seen that the limit in distribution of the family of measures $ \CondDistVoneVtwoeps $ is given by $\nu_{*}^{ \otimes k}$, where $\nu_{ *}$ is given as
\begin{equation}
\label{EqMaxEntrPrinc}
\nu_{*} = \arg \max \big \{ h( \nu ) : m_{V_1}( \nu ) \leq 1, m_{V_2}( \nu ) \in [0, R ] \big  \}.
\end{equation}
By Lemma \ref{LemMonotonepMoment}, there exists a $\overline{R} \in (0, \infty )$ such that $ m_{V_1}( \mu_{ V_2, \alpha( \overline{R})})= 1$ and $ m_{V_2}( \mu_{ V_2, \alpha( \overline{R})})= \overline{R} $. We claim that for $0 < R \leq \overline{R}$, the maximizing measure in \eqref{EqMaxEntrPrinc} is the Gibbs distribution $\mu_{V_2, \alpha(R)}$ with Lebesgue density
\begin{equation*}
 \frac{d \mu_{V_2, \alpha(R)}}{dx}(x) = e^{ \alpha(R) V_2(x) - \varphi_{V_2}( \alpha(R) , 0)}, \quad x \in \R .
\end{equation*} 
First, we observe that the density of $ \mu_{V_2, \alpha(R)}  $ is of the form in Equation \eqref{EqEntrResprMaxMeas} from Lemma \ref{LemMaxEntr} (with the choice of $n=m=1$, $\kappa^{*}_0=1+ \varphi_{V_2}( \alpha(R) , 0) \in \R $, $ \lambda_1^{*} = - \alpha(R)$ and $ \mu_1^{*} = 0$). 
By assumption, we have that $R \leq \overline{R}$ and hence $m_{V_1}(\mu_{V_2, \alpha(R)}) \leq 1$ (see Equation \eqref{EqIneqpthMoment} in Lemma \ref{LemMonotonepMoment}) and by the choice of $\alpha(R)$, we get that $ m_{V_2}( \mu_{V_2, \alpha(R)} )= R$. Hence, by Lemma \ref{LemMaxEntr} we have that $\nu_{*} = \mu_{V_2, \alpha(R)}$. Let $d_{ w}$ be some metric inducing weak convergence of probability measures (e.g., the bounded Lipschitz metric). Then, by Proposition \ref{PropLimUnifOrlicz} and Remark \ref{RemConDistOrlicGibbs},
\begin{equation*}
d_{ w} \big ( \CondDistVoneVtwoeps , \lambda_{R,V_2}^{(n,k)} \big ) \leq d_{ w} \big ( \CondDistVoneVtwoeps , \mu_{V_2, \alpha(R)}^{ \otimes k } \big ) + d_{ w} \big ( \mu_{V_2, \alpha(R)}^{ \otimes k } ,\lambda_{R,V_2}^{(n,k)} \big ) \longrightarrow 0
\end{equation*}
as $n \rightarrow \infty$ followed by $ \eps \rightarrow 0$. 
\end{proof}

At this point one might ask what happens in the other cases, i.e., when the quantity $R \in (0,\infty)$ in the definition of the Orlicz ball $ \mathbb{B}^{n,V_2}_R$ is \glqq{}big\grqq. This leads to the following corollary of Proposition \ref{PropMainRes}.

\begin{cor}
\label{CorBigVtwoRadius}
Let $V_1, V_2$ be two Orlicz functions that satisfy
\begin{equation*}
\frac{V_1(x)}{V_2(x)} \longrightarrow \infty \quad \text{as} \quad |x| \rightarrow \infty. 
\end{equation*}
Take some $ R \in [\widetilde{R}, \infty)$ with
\begin{equation}
\label{DefRtilde}
\widetilde{R} := \sup_{ \nu \in K_{V_1}^1} m_{V_2}(\nu ) ,
\end{equation}
where $K_{V_1}^1 = m_{V_1}^{-1}( [0,1])$ is the set from Proposition \ref{PropCompactSetGenWasserst}.
Then, for fixed $k \in \N$, we have the weak limit
\begin{equation*}
\lim_{ \eps \rightarrow 0} \lim_{n \rightarrow \infty } d_w ( \CondDistVoneVtwoeps, \lambda_{1 , V_1}^{(n,k)}) = 0,
\end{equation*}
where $d_w$ is some metric inducing weak convergence of probability measures and $ \lambda_{1 , V_1}^{(n,k)}$ is the distribution of the first $k$ components of a uniform distribution on $\mathbb{B}^{n, V_1}_1$. 
\end{cor} 

\begin{rem}
By Proposition \ref{PropCompactSetGenWasserst}, $K_{V_1}^1$ is compact in $(\mathscr{M}^{V_2}_1( \R ), d_{V_2} )$ and $m_{V_2}:\mathscr{M}^{V_2}_1( \R ) \rightarrow \R$ is continuous, which implies that $ \widetilde{R} $ from Equation \eqref{DefRtilde} is finite.
\end{rem}

\begin{proof}(of Corollary \ref{CorBigVtwoRadius})
By Proposition \ref{PropMainRes}, we have that the limit in distribution, as $n \rightarrow \infty$ followed by $\eps \rightarrow 0$, of $\CondDistVoneVtwoeps$ is given by $ \nu_{*}^{ \otimes k}$, where 
\begin{equation*}
\nu_{*}:= \arg \max \big \{ h ( \nu ) : \nu \in 
\mathscr{M}^{V_2}_1( \R ) , m_{V_1}( \nu ) \leq 1, m_{V_2}( \nu ) \leq R \big  \}.
\end{equation*}
Since we assume $R \geq \widetilde{R}$, we see that  for each $\nu \in K_{V_1}^1$, the inequality $m_{V_1}( \nu ) \leq 1$ implies $m_{V_2}( \nu ) \leq R $. Thus $\nu_{*}$ is of the form $ \nu_{*} = \arg \max \big \{ h ( \nu ) : \nu \in \mathscr{M}^{V_2}_1( \R ) , m_{V_1}( \nu ) \leq 1 \big \}$, where the maximum is attained by the Gibbs measure $ \mu_{ V_1, \alpha(1)}$ (see proof of Theorem \ref{ThmLimCondDist}). By Proposition \ref{PropLimUnifOrlicz}, we have that $ \lambda_{1,V_1}^{(n,k)} $ converges in distribution to $ \mu_{ V_1, \alpha(1)}^{ \otimes k} $ as $n \rightarrow \infty$ and thus
\begin{equation*}
\lim_{ \eps \rightarrow 0} \lim_{n \rightarrow \infty } d_w ( \CondDistVoneVtwoeps, \lambda_{n , V_1}^{(1,k)}) \leq \lim_{ \eps \rightarrow 0} \lim_{n \rightarrow \infty } d_w ( \CondDistVoneVtwoeps, \mu_{ V_1, \alpha(1)}^{ \otimes k}) + \lim_{n \rightarrow \infty } d_w ( \mu_{ V_1, \alpha(1)}^{ \otimes k}, \lambda_{n , V_1}^{(1,k)}) = 0,
\end{equation*}
which completes the proof.
\end{proof}

\begin{rem}
	\label{RemInterMedCaseOrliczRadius}
	In the \glqq{}intermediate\grqq{} case, where $R \in (\overline{R}, \widetilde{R})$ with $\overline{R}$ being the quantity from Theorem \ref{ThmLimCondDist} and $\widetilde{R}$ from Equation \eqref{DefRtilde}, we get that the limit in distribution of $ \CondDistVoneVtwoeps$ is given by $ \nu_{*}^{ \otimes k}$. $\nu_{*}$ is some distribution proportional to
	\begin{equation*}
	e^{- \mu_1^{*} V_1(x) - \mu_2^{*} V_2(x) }, \quad x \in \R .
	\end{equation*} 
	The constants $ \mu_i^{*} \geq 0$ follow from the maximum entropy principle Lemma \ref{LemMaxEntr}. In total, one can say that the limiting distribution of $ \CondDistVoneVtwoeps $ undergoes a phase transition depending on the magnitude of the parameter $R$.
\end{rem}

\subsection*{Acknowledgement}
LF and JP are supported by the Austrian Science Fund (FWF) Project P32405 \textit{Asymptotic geometric analysis and applications} and FWF Project F5513-N26, which is a part of the Special Research Program \textit{Quasi-Monte Carlo Methods: Theory and Applications}.

\bibliographystyle{plain}
\bibliography{CondLimThm_bib.bib}

\bigskip
\bigskip

\medskip

\small

\noindent \textsc{Lorenz Fr\"uhwirth:} Institute of Mathematics and Scientific Computing,
University of Graz, Heinrichstra{\ss}e 36, 8010 Graz, Austria

\noindent
\textit{E-mail:} \texttt{lorenz.fruehwirth@uni-graz.at}

\medskip

\small

\noindent \textsc{Joscha Prochno:} Faculty of Computer Science and Mathematics, University of Passau, Innstra{\ss}e 33, 94032 Passau, Germany. 

\noindent
\textit{E-mail:} \texttt{joscha.prochno@uni-passau.de}

\end{document}